\documentclass[12pt]{article}

\usepackage{amssymb}   
\usepackage{amsthm}    
\usepackage{amsmath}   
\usepackage{stmaryrd}  
\usepackage{titletoc}  
\usepackage{mathrsfs}  
\usepackage{graphicx}

\newlength{\defbaselineskip}
\newcommand{\setlinespacing}[1]%
           {\setlength{\baselineskip}{#1 \defbaselineskip}}

\theoremstyle{plain}
\newtheorem{thm}{Theorem}[section]
\newtheorem{cor}[thm]{Corollary}
\newtheorem{lem}[thm]{Lemma}
\newtheorem{prop}[thm]{Proposition}

\theoremstyle{definition}
\newtheorem{defn}{Definition}[section]
\newtheorem{ass}{Assumption}[section]
\newtheorem{rmk}{Remark}[section]

\newcommand{\eps}{\varepsilon}

\newcommand{\cL}{\mathcal{L}}
\newcommand{\cM}{\mathcal{M}}
\newcommand{\cB}{\mathcal{B}}
\newcommand{\cS}{\mathcal{S}}
\newcommand{\cF}{\mathcal {F}}

\newcommand{\fS}{\mathfrak{S}}
\newcommand{\bH}{\mathbb{H}}
\newcommand{\bP}{\mathbb{P}}
\newcommand{\bR}{\mathbb{R}}
\newcommand{\bD}{\mathbb{D}}
\newcommand{\sF}{\mathscr{F}}
\newcommand{\sP}{\mathscr{P}}
\newcommand{\sD}{\mathscr{D}}
\newcommand{\sH}{\mathscr{H}}
\newcommand{\sS}{\mathscr{S}}

\newcommand{\rrow}{\rightarrow}

\makeatletter\@addtoreset{equation}{section} \makeatother
\begin{document}

\title{$L^{p}$ Theory for Super-parabolic Backward Stochastic
Partial Differential Equations in the Whole Space\footnotemark[1]}

\author{Kai Du\footnotemark[2],  \quad Jinniao
Qiu\footnotemark[2],\quad  and\quad Shanjian
Tang\footnotemark[2]~\footnotemark[3] }

\footnotetext[1]{Supported by NSFC Grant \#10325101, by Basic
Research Program of China (973 Program)  Grant \# 2007CB814904, by
the Science Foundation of the Ministry of Education of China Grant
\#200900071110001, and by WCU (World Class University) Program
through the Korea Science and Engineering Foundation funded by the
Ministry of Education, Science and Technology (R31-2009-000-20007).}

\footnotetext[2]{Department of Finance and Control Sciences, School of
Mathematical Sciences, Fudan University, Shanghai 200433, China.
\textit{E-mail}: \texttt{kdu@fudan.edu.cn} (Kai Du),
\texttt{071018032@fudan.edu.cn} (Jinniao Qiu), \texttt{sjtang@fudan.edu.cn}
(Shanjian Tang).}

\footnotetext[3]{Graduate Department of Financial Engineering, Ajou
University, San 5, Woncheon-dong, Yeongtong-gu, Suwon, 443-749,
Korea.}

\maketitle

\begin{abstract}
This paper is concerned with semi-linear backward stochastic partial
differential equations (BSPDEs for short) of super-parabolic type.
An $L^p$-theory  is given for the Cauchy problem of BSPDEs,
separately for  the case of $p\in (1,2]$ and for the case of $p\in
(2, \infty)$. A comparison theorem is also addressed.
\end{abstract}



\section{Introduction}
Since Bismut's pioneering work \cite{Bismut_76,Bismut_78,Bismut_78_2} and
Pardoux and Peng's seminal work \cite{ParPeng_90}, the theory of backward
 stochastic differential equations (BSDEs) is rather complete now. See, among others,
 El Karoui et al.~\cite{Karoui_Peng_Quenez}, and Delbaen and Tang
 \cite{DelbaenTang_10} for a rather general $L^p$ theory for BSDEs.
  As a natural generalization of BSDEs,
  backward stochastic partial differential equations (BSPDEs) arise
in many applications of probability theory and stochastic processes, for
instance in the optimal control of processes with incomplete information,
as an adjoint equation of the Duncan-Mortensen-Zakai filtration equation
(for instance, see
\cite{Bensousan_83,Hu_Ma_Yong02,Hu_Peng_91,Tang_98,Zhou_92,Zhou_93}), and
naturally in the dynamic programming theory fully nonlinear BSPDEs as the
so-called  backward stochastic Hamilton-Jacobi-Bellman equations, are also
introduced in the study of controlled non-Markovian processes (see Peng
\cite{Peng_92} and Englezos and Karatzas \cite{EnglezosKaratzas09}).

In this paper, we consider the following semi-linear BSPDEs:
\begin{equation}\label{BSPDE intro 1}
  \left\{\begin{array}{l}
  \begin{split}
  -du(t,x)=&[\cL(t,x)u(t,x)+\cM^r(t,x) v^r (t,x)+F(u,v,t,x)]dt  \\
           &~-v^{r}(t,x)dW_{t}^{r}, \quad
                     (t,x)\in[0,T]\times\bR^d;\\
    u(T,x)=&G(x), \quad x\in\bR^d.
    \end{split}
  \end{array}\right.
\end{equation}
Here and throughout this paper, we denote
 $$   \cL(t,x):= a^{ij}(t,x)\frac{\partial^2}{\partial x^i\partial
 x^j},\quad
    \cM^r(t,x):=\sigma^{jr}(t,x)\frac{\partial}{\partial x^j}, \quad
r=1,2,\dots,m. $$ We use the Einstein summation convention and fix
$T\in(0,\infty)$ as a finite
 deterministic time, which can be replaced by any bounded stopping time.

To the above BSPDEs, the method of stochastic flows was developed by
Tang \cite{Tang_05} which gives a probabilistic point of view and
also gives classical solutions to BSPDEs~\eqref{BSPDE intro 1}. On
the other hand, the $L^2$ theory for BSPDEs has been established in
the framework of weak solutions (see
\cite{DuMeng09,Hu_Ma_Yong02,Hu_Peng_91,Zhou_92,Zhou_93}, for
example).

Still in the framework of weak solutions, we establish in this paper an
$L^p$-theory for BSPDE \eqref{BSPDE intro 1} which seems to be the first
study for the $L^p$-theory of BSPDEs. Motivated by Krylov's semianl work
\cite{Kryl96,Krylov_99} on forward stochastic partial differential
equations, we consider BSPDE as the generalized backward Kolmogorov
equation and establish an $L^p$-theory which includes as a particular case
the $L^p$ theory $(1<p\leq 2)$ for deterministic parabolic partial
differential equations (PDEs for short).

Our results are based on the duality between BSPDEs and stochastic partial
differential equations (SPDEs). In response to the requirement that $p\geq
2$ in the $L^p$ theory of SPDEs established by Krylov
\cite{Kryl96,Krylov_99} we require $p\in (1,2]$ in our $L^p$ theory for
BSPDEs.

This paper is organized as follows. In Section 2 we introduce the
notions and define some spaces. We discuss a kind of Banach
space-valued BSDEs in Section 3. In Section 4 we construct a
stochastic Banach space $\sH_p^n$ which plays the same role as
spaces $W^{1,2}_p$ in the theory of second-order parabolic PDEs and
we also give some basic properties of this space there. In Section 5
we present the $L^p$-theory of BSPDEs in the whole space for $p\in
(1,2]$. Specifically, we give the definition of the $L^p$ solutions
and list the assumptions. We first solve the BSPDEs with
constant-field-valued leading coefficients and then solve the BSPDEs
for the general case. In Section 6 we discus two related topics: a
comparison theorem and an $L^p$-theory for $p>2$.

\section{Preliminaries}
In most of this work, we shall denote by $|\cdot|$ (respectively,
$<\cdot,\cdot>$) the norm (respectively, scalar product) in
finite-dimension Hilbert space such as $\bR,\bR^k,\bR^{k\times l}$ where
$k,l$ are positive integers and
$$|x|:=\left( \sum_{i=1}^k x_i^2  \right)^{\frac{1}{2}} \quad
\textrm{and}\quad  |y|:=\left(\sum_{i=1}^k\sum_{j=1}^l y_{ij}^2
\right)^{\frac{1}{2}} \quad \textrm{for}~ (x,y)\in \bR^k \times
\bR^{k\times l}.
$$
Let $(\Omega,\sF,\{\sF_t\}_{t\geq0},\bP)$ be a complete filtered
probability space on which is defined a $m$-dimensional standard Brownian
motion $W=\{W(t):t\in[0,T]\}$ such that $\{\sF_t\}_{t\geq0}$ is the natural
filtration generated by $W$ and augmented by all the $\bP$-null sets in
$\sF$. And we denote by $\sP$ the $\sigma$-Algebra of the predictable sets
on $\Omega\times[0,T]$ associated with $\{\sF_t\}_{t\geq0}$.

If $X=(X_{t})_{t\in [0,T]}$ is an $\bR^{k}$-valued, adapted and
 continuous processes, we denote by $X_{*}$ or $\sup_t |X_t|$
 where $|\cdot|$ denotes the Euclidean norm on $\bR^{k}.$ And for any $p\in
 (1,\infty)$, $\cS ^p (\bR^k)$ denotes the set of all the $\bR^k$-valued,
 adapted and continuous processes $(X_{t})_{t\in [0,T]}$ such
 that
 $$\|X\|_{\cS ^p}:= E [\sup_t |X_t|^p] < \infty.$$

We denote by $C^{\infty}_{c}$ the set of all infinitely
differentiable functions of compact supports on  $\bR^{d}$ and by
$\mathscr{D}$ the space of real-valued Schwartz distributions on
$C^{\infty}_{c}$. And also, on $\bR^d$ we denote by $\sS$ the set of
all the Schwartz functions and by $\sS'$ the set of all the tempered
distributions. Note that $C^{\infty}_{c}$ and $\sS$ are endowed with
matching topologies (see, for instance \cite{Loukas_book}). We shall
denote by $(\cdot,\cdot)$ not only the duality between $\mathscr{D}$
and $C^{\infty}_{c}$ but also the duality between $\sS$ and $\sS'.$
Then the Fourier transform $\cF(f)$ of $f\in \sS'$ is given by
$$\cF(f)(\xi)=(2\pi)^{-d/2}\int_{\bR^d}e^{-\sqrt{-1}<x,\xi>}f(x)dx,
~~~\xi\in \bR^d,$$ and the inverse Fourier transform $\cF^{-1}(f)$ is given
by
$$
    \cF(f)(x)=(2\pi)^{-d/2}\int_{\bR^d}e^{\sqrt{-1}<x,\xi>}f(\xi)d\xi,
~~~x\in \bR^d. $$
 It is well known that both $\cF$ and $\cF^{-1}$ map $\sS'$ onto itself. As
usual, for any $s\in \bR$ and $f\in \sS'$, we denote
$I_{s}(f):=(1-\Delta)^{s/2}f=\cF^{-1}((1+|\xi|^2)^{s/2}\cF(f)(\xi)).$

 For given $p\in(1,\infty)$ and $n\in
(-\infty,\infty)$, we denote by $H ^{n}_{p}$ the space of Bessel
potentials, that is
$$H ^{n}_{p}:=\{\phi\in \sS':(1-\Delta)^{\frac{n}{2}}\phi\in L^{p}(\bR ^{d})\}$$
with the Sobolev norm
$$\|\phi\|_{n,p}:=\|(1-\Delta)^{\frac{n}{2}}\phi\|_{p}, ~~ \phi\in
H^{n}_{p},$$ where $\|\cdot\|_{p}$ is the norm in $L^{p}(\bR^d)$. It
is well known that $H ^{n}_{p}$ is a Banach space with the norm
$\|\cdot\|_{n,p}$ and the set $C^{\infty}_{c}$ is dense in $H
^{n}_{p}$. For any $p\in(1,\infty)$ and $n\in \bR,$ we denote by
$(\cdot,\cdot)$ the dual pairing between $H_p^n$ and $H_{p'}^{-n}$
where $1/p'+1/p=1,$ i.e., for any $(u, v)\in H^{n}_{p}\times
H^{-n}_{p'}$
$$(u,v)=((1-\Delta)^{\frac{n}{2}}u,(1-\Delta)^{-\frac{n}{2}}v)
    =\int_{\bR^{d}}(1-\Delta)^{\frac{n}{2}}u(x)(1-\Delta)^{-\frac{n}{2}} v(x)\, dx$$
where the last integral is a usual Lebesgue integral.

Define the set of multi-indices
$$\mathcal {A}:=\{\alpha=(\alpha_1,\dots,\alpha_d): \alpha_1, \dots, \alpha_d \textrm{
are nonnegative integers}\}.$$ For any $\alpha\in \mathcal{A}$ and
$x=(x^1,\dots,x^d)\in \bR^d,$ denote
$$ |\alpha|=\sum_{i=1}^d
\alpha_i,~~~~D^{\alpha}:=\frac{\partial^{|\alpha|}}
{\partial x^{\alpha_1}\partial x^{\alpha_2}\cdots\partial x^{\alpha_d}}. $$

In contrast to $H_p^n,$ we introduce the following so-called Besov
space of functions (c.f \cite{TRiebel_83} or \cite{Triebel_92}).

\begin{defn}
  Let $s>0,$ $p\in (1,\infty),$ and $q\in [1,\infty).$ Define
  \begin{eqnarray*}
  \begin{split}
    B_{p,q}^s=&\biggl\{ f\in L^p(\bR^d):\|f\|_{B_{p,q}^s}=\|f\|_{H^{[s]^-}_p}        \\
            &+\sum_{|\alpha|=[s]^-}\left(\int_{\bR^d}
            |h|^{-\{s\}^+q}\|D^{\alpha}f(\cdot+2h)
            -2D^{\alpha}f(\cdot+h)+D^{\alpha}f(\cdot)\|_p ^q \frac{dh}{|h|^d}
            \right)^{1/q}<\infty  \biggr\}
  \end{split}
\end{eqnarray*}
where $s=[s]^-+\{s\}^+,$ $[s]^-$ is an integer and $\{s\}^+\in
(0,1].$
\end{defn}
Let $\sigma>0,$ $p\in (1,\infty),$ $q\in [1,\infty),$ and $s\in\bR$ such
that $\sigma-s>0.$ Then $I_s (B_{p,q}^{\sigma})=B_{p,q}^{\sigma-s}.$ In
fact, we can introduce spaces $B_{p,q}^s$ with $s\leq 0$ by defining
$B_{p,q}^s=I_{-s+1}(B_{p,q}^1),$ although we prefer to define the Besov
space through the Littlewood-Paley decomposition (for instance, see
\cite{Triebel_92}). As to the specific structure and properties of Besov
space, see \cite{Triebel_92} or \cite{TRiebel_83}. In this paper,
 only the space $B_{p,p}^n$ is involved for $p\in(1,\infty)$ and $n\in
\bR.$

Denote by $\fS$ the set of all $\sS'$-valued functions defined on
$\Omega\times [0,T]$ such that, for any $u\in \fS$ and $\phi\in \sS,$ the
function $(u,\phi)$ is $\sP$-measurable.

For $p\in(1,\infty),$ we define
$\bH_p^0:=L^p(\Omega\times[0,T]\times\bR^d,\sP\times\cB(\bR^d),\bR).$
 Denote by $\bH_{p,2}^0$ the set of the functions which are defined on
 $\Omega\times[0,T]\times\bR^d$ and  $\sP\times\cB(\bR^d)$-measurable such
 that $$ E\left[\int_{\bR^d}\left(\int_{0}^{T}
    |u(t,x)|^{2}dt\right ) ^{\frac{p}{2}}dx\right]<\infty,~~\forall~ u\in\bH_{p,2}^0.$$

Observe that every element of $\bH_p^0$ can be considered as an
$H_p^0$-valued, $\sP$-measurable process. For any $n\in\bR,$ we define
$$\bH_p^n=\{f\in \fS:~(1-\Delta)^{\frac{n}{2}}f\in \bH_p^0 \},$$
equipped with the norm
$$\|f\|_{\bH^n_p}:=\left(E\left[ \int_0^T \int_{\bR^d}|
(1-\Delta)^{\frac{n}{2}}f(t,x)|^pdxdt\right]\right)^{1/p}.$$

\begin{defn}
  Let $p\in (1,\infty)$ and $n\in \bR.$ Define
  $$\bH_{p,2}^{n}=\left\{u\in\fS:~(1-\Delta)^{\frac{n}{2}}u\in\bH_{p,0}^0   \right\}  $$
    equipped with the norm
    $$\|u\|_{\bH_{p,2}^{n}}:=\left(E\left[\int_{\bR^d}\left[\int_{0}^{T}
    |(1-\Delta)^{\frac{n}{2}}u(t,x)|^{2}dt\right ] ^{\frac{p}{2}}dx\right]\right)^{1/p}
    .$$
\end{defn}
\begin{defn}
    Let $p\in (1,\infty)$ and $n\in \bR.$ For a function $u\in
    \bH_{p,2}^{n},$ we write $u\in\bH_{p,\infty}^n$ if

    (i) there exists $A(u)\in\sF_T\times\cB(\bR^d)$,
    $\bP\times\mathfrak{M}(A(u))=0$ where $\mathfrak{M}(\cdot)$
       denotes the Lebesgue measure on $\bR^d$, such that
        for any
        $(\omega,x)\in \bR^d\times\Omega \setminus A(u),$
        $(1-\Delta)^{n/2}u(\cdot,x)$ is continuous on $[0,T];$

    (ii) $\|u\|_{\bH_{p,\infty}^{n}}:=
        \big(E\big[\int_{\bR^{d}}\sup_{t\in[0,T]}
        |(1-\Delta)^{\frac{n}{2}}u(t,x)|^{p}dx\big]\big)^{1/p}<\infty.$
\end{defn}

When we treat the general $\bR^k$-valued function $u$ for  any
integer $k>1$, we
 still say $u\in \bH_p^{n}$ if $u^l\in \bH_p^n$ for $l=1,2,\dots,k.$
 In this way, we generalize the real-valued function space $\bH_p^n$ to
 $\bR^k$-valued function space. And further, we define the norm
 $$\|u\|_{\bH_p^n}=\left(E\left[ \int_0^T \int_{\bR^d}|
(1-\Delta)^{\frac{n}{2}}u(t,x)|^pdxdt\right]\right)^{1/p}.$$

By this means, not only can we generalize spaces $H_p^n,$ $\bH_p^n,$
$\bH_{p,2}^n$ and $\bH_{p,\infty}^n$ from real-valued function
spaces to any $\bR^k$-valued ones, but also we can generalize these
spaces from $\bR^k$-valued to any Hilbert space-valued function
spaces. And we do it when we need it.

\begin{rmk}\label{rmk setion 1}
    One can check that the spaces $\bH_p^n,~\bH_{p,2}^n$ and
    $\bH_{p,\infty}^n$ are all Banach spaces under the norms
    $\|\cdot\|_{\bH_p^n},$ $\|\cdot\|_{\bH_{p,2}^n},$ and
    $\|\cdot\|_{\bH_{p,\infty}^n},$ respectively. Moreover, for any $p\in (1,\infty)$ and
    $n\in \bR,$ $\bH_p^n$
    is a reflexive Banach space whose dual space is $\bH_{p/(p-1)}^{-n}$, and it coincides
    with the space $\bH_p^n(T)$ defined in \cite{Krylov_99} and
     \cite{Kryl96}. On the other hand, for $s\in\bR,$ the operator $(1-\Delta)^{s/2}$
    maps isometrically $H_p^n$ to $H_p^{n-s}$ and the same is true
    for spaces $\bH_p^n,~\bH_{p,2}^n,$ and $\bH_{p,\infty}^n.$
\end{rmk}
 In particular, as to the spaces $\bH_p^n$ and $\bH_{p,2}^n,$
 we have the following lemma£¬ whose proof is similar to that of
 \cite[Theorem 3.10]{Krylov_99}.

\begin{lem}\label{lem finite approximation H_pn} Let $p\in (1,
\infty)$ and $n\in \bR$. For $g\in\bH ^{n}_{p}$ ($\bH_{p,2}^n$,
respectively), there exits a sequence $\{g_{j},~j=1,2\dots\}$
    in $\bH^{n}_{p}$ ($\bH_{p,2}^n$, respectively)
     such that $\|g-g_{j}\|_{\bH ^{n}_{p}}\rightarrow
    0$ ($\|g-g_{j}\|_{\bH ^{n}_{p,2}}\rightarrow 0 ,$ respectively) as $j\rightarrow \infty$ and
$$g_{j}= \sum_{i=1}^{j}
         \mathbb{I} _{(\tau _{i-1}^{j},\tau _{i}^{j}]}(t)g_{j}^{i}(x),$$
where $g_{j}^{i}\in C^{\infty}_{c}$ and $\tau _{i}^{j}$
   are stopping times such that $\tau _{i-1}^{j}\leq\tau
 _{i}^{j}\leq T.$
\end{lem}

For any $t\in[0,T)$, define
$$\|u\|_{\bH_p^n(t)}=\|u\mathbb{I}_{[t,T]}\|_{\bH_p^n}
\quad \textrm{ for }u\in \bH_p^n.$$ In the same way, we define
$\|\cdot\|_{\bH_{p,2}^n(t)}$ in $\bH_{p,2}^n$ and
$\|\cdot\|_{\bH_{p,\infty}^n(t)}$
 in $\bH_{p,\infty}^n$.

%
For an element $u$ of spaces like $\bH_p^n$, if it has a
modification of higher regularity, then it is always considered to
be this modification. However, elements of spaces like $\bH_p^n$
belong to $H_p^n$  only for almost all $(t,\omega)$, not necessarily
for all $(t,\omega)\in [0,T]\times\Omega$.

\section{Banach space-valued BSDEs}

This section is concerned with Banach space-valued BSDEs. Unless stated
otherwise, we assume $p\in (1,\infty)$ and $n\in \bR$ throughout this
section.
 For $(F,G)\in\bH
^{n}_{p}\times L^{p}(\Omega,\sF_{T},H_{p}^{n})$, consider the BSDE
\begin{equation}\label{bsde}
  \left\{\begin{array}{l}
    \begin{split}
      -du(t,x)=F(t,x)dt-v^{k}(t,x)dW_t^{k},~~~(t,x)\in[0,T]\times \bR^{d},
    \end{split}\\
    \begin{split}
      u(T,x)=G(x),~~~~~~~~~x\in \bR^{d}.
    \end{split}
  \end{array}\right.
\end{equation}
Or, equivalently
$$u(t,x)=G(x)+\int_{t}^{T}F(s,x)ds-\int_{t}^{T}v^{k}(s,x)dW^{k}_{s},
        ~~(t,x)\in[0,T]\times\bR^{d}.$$

\begin{defn}\label{S bsde}
   Assume that $(F,G)\in\bH
  ^{n}_{p}\times L^{p}(\Omega,\sF_{T},H_{p}^{n})$ with $p\in (1,
\infty)$ and $n\in \bR$. We say $(u,v)\in\bH_{p}^{n}\times
\bH_{p,2}^{n}$ is a
   solution of \eqref{bsde} if for any $ \phi \in C^{\infty}_{c}$
   and $\tau\in[0, T],$ we have
\begin{eqnarray}\label{S bsde1}
    \begin{split}
    (u(\tau,\cdot),\phi)=(G,\phi)+
            \int_{\tau}^{T}(F(s,\cdot),\phi)\, ds
                    -\int_{\tau}^{T}(v^{l}(s,\cdot),\phi)\, dW^{l}_{s},\quad a.s..
  \end{split}
\end{eqnarray}
\end{defn}

\begin{rmk}\label{rmk of bsde}
    If $(u,v)\in\bH_{p}^{n}\times \bH_{p,2}^{n}$ is a
   solution to \eqref{bsde}, then for any $\phi\in H_{p'}^{-n},$
\begin{eqnarray*}
  \begin{split}
   &E\left[\max_{t\in [0,T]}|\int_0^T(v^l,\phi)\, dW_s^l  |\right]
    \\
   \leq & C
     E\left[\left( \int_0^T|(v(s,\cdot),\phi)|^2\, ds \right)^{1/2}  \right]
     \quad (\textrm{using the BDG inequality})   \\
    =& C E\left[\left( \int_0^T|
     \int_{\bR^d}(1-\Delta)^{-n/2}\phi(x)(1-\Delta)^{n/2}v(s,x)dx
     |^2\, ds \right)^{1/2}  \right]\\
    &(\textrm{using Minkowski inequality}) \\
     \leq &C E\left[ \int_{\bR^d}\left(\int_0^T|
     (1-\Delta)^{-n/2}\phi(x)(1-\Delta)^{n/2}v(s,x)
     |^2ds \right)^{1/2}dx  \right]\\
    =&C E\left[ \int_{\bR^d}\left(\int_0^T|
     (1-\Delta)^{n/2}v(s,x)
     |^2ds \right)^{1/2}(1-\Delta)^{-n/2}\phi(x)\, dx  \right]\\
    \leq &C \|v\|_{\bH_{p,2}^{n}}\|\phi\|_{-n,p'}
  \end{split}
\end{eqnarray*}
where $p'+p=1.$ So, the process
$$\int_0^t(v^l(s,\cdot),\phi)dW^l_s
  , \quad t\in[0,T]$$ is a continuous martingale.
  Note that, throughout the paper, unless stated otherwise,
    $C$ is a positive
  constant and  $C(\alpha,\beta,\cdots,\gamma)$ is a  constant
     only depending on $\alpha,\beta,\cdots,$ and $\gamma$.
\end{rmk}



\begin{lem}\label{lem BSDE}
  Assume that $(F,G)\in\bH^{n}_{p}\times
  L^{p}(\Omega,\sF_{T},H_{p}^{n})$ with $p\in (1,
\infty)$ and $n\in \bR$.
  We have

  (i) equation \eqref{bsde} has a unique solution $(u,v)\in(\bH_{p}^{n}\cap
  \bH_{p,\infty}^{n})\times\bH_{p,2}^{n}$ which satisfies the following
  inequality
    \begin{equation}\label{Estimate of bsde}
      \|u\|_{\bH_{p,\infty}^{n}}+ \|u\|_{\bH_p^n} +\|v\|_{\bH_{p,2}^{n}}\leq
      c(p,T)[\|F\|_{\bH^{n}_{p}}+\|G\|_{L^{p}(\Omega,\sF_{T},H_{p}^{n})}
      ].
    \end{equation}

    (ii) For this solution, we have $u\in C([0,T],H_p^n)$ almost
    surely, and for any $ \phi \in H_{p/(p-1)}^{-n}$  the
following equality
\begin{eqnarray}\label{S bsde2}
    \begin{split}
    (u(\tau,\cdot),\phi)=(G,\phi)+
            \int_{\tau}^{T}(F(s,\cdot),\phi)ds
                    -\int_{\tau}^{T}(v^{l}(s,\cdot),\phi)dW^{l}_{s}
  \end{split}
\end{eqnarray}
holds for all $\tau\in[0, T]$ with probability 1.
\end{lem}



\begin{proof}
First, we prove the uniqueness of the solution. Suppose that
$(u_1,v_1)$ and $(u_2,v_2)$ are two solutions of \eqref{bsde} in
$\bH_p^n\times \bH_{p,2}^n,$ and take $(u,v)=(u_1-u_2,v_1-v_2).$
Then, for any $\phi\in C^{\infty}_{c}$ and $t\in [0,T]$ We have
$$(u(t,\cdot),\phi)=\int_t^T (v(s,\cdot),\phi)\, dW_s,\quad ~ a.s..$$
Then by the theory on BSDEs (c.f.
\cite{Hu_2002,Karoui_Peng_Quenez,ParPeng_90}), we have
$$E\left[\int_{\tau_1}^{\tau_2}(u(t,\cdot),\phi)\, dt\right]=0\textrm{ and }
E\left[\int_{\tau_1}^{\tau_2}(v(s,\cdot),\phi)\, ds\right]=0, $$
 for any
stopping times $\tau_1\textit{ and }\tau_2,$ $0\leq \tau_1\leq \tau_2\leq
T.$ From Lemma \ref{lem finite approximation H_pn},
it follows that $(u,v)=0$ in
$\bH_p^n\times\bH_{p,2}^n.$ This verifies the uniqueness.

 For the other assertions, it is sufficient to prove the lemma for $n=0$.

Indeed, assume that the lemma is true for $n=k$ with $k\in \bR$ For
$\forall \delta\in \bR$, if $(F,G)\in \bH^{k+\delta}_{p}\times
L^{p}(\Omega,\sF_{T},H_{p}^{k+\delta})$ then $(F',G')\in
\bH^{k}_{p}\times L^{p}(\Omega,\sF_{T},H_{p}^{k})$ with
$$F'
:=(1-\Delta)^{\delta/2}F\quad \textrm{and}\quad
G':=(1-\Delta)^{\delta/2}G.$$ From the induction assumption, there
exists $(u',v')\in \bH_{p,\infty}^{k}\times \bH_{p,2}^{k}$,
satisfying the following
\begin{equation}\label{InductionEst}
      \|u'\|_{\bH_{p,\infty}^{k}}+ \|u'\|_{\bH_p^k} +\|v'\|_{\bH_{p,2}^{k}}\leq
      c(p,T)\left(\|F'\|_{\bH^{k}_{p}}+\|G'\|_{L^{p}(\Omega,\sF_{T},H_{p}^{k})}
\right),
    \end{equation}
such that for any $\phi\in H_{p/(p-1)}^{-k}$ the equality
$$(u'(\tau,\cdot),\phi)=(G',\phi)+
            \int_{\tau}^{T}(F'(s,\cdot),\phi)\, ds
                    -\int_{\tau}^{T}(v'^{l}(s,\cdot),\phi)\, dW^{l}_{s},$$
 holds for all $\tau\in[0,T]$ with probability 1.
 Take
 $$u=(1-\Delta)^{-\delta/2}u'\quad \textrm{and}\quad
 v^{l}=(1-\Delta)^{-\delta/2}v'^{l}.$$
 In view of Remark \ref{rmk setion 1}, $(u,v)\in
 \bH^{k+\delta}_{p}\times L^{p}(\Omega,\sF_{T},H_{p}^{k+\delta})$. Rewrite the last equality
 into the following
 \begin{eqnarray*}
  \begin{split}
((1-\Delta)^{\delta/2}u(\tau,\cdot),\phi)=&((1-\Delta)^{\delta/2}G,\phi)+
            \int_{\tau}^{T}((1-\Delta)^{\delta/2}F(s,\cdot),\phi)\, ds\\
 &-\int_{\tau}^{T}((1-\Delta)^{\delta/2}v^{l}(s,\cdot),\phi)\,
 dW^{l}_{s}, \quad \phi\in
H_{p/(p-1)}^{-k}
  \end{split}
\end{eqnarray*}
which is equivalent to
 \begin{eqnarray*}
  \begin{split}
(u(\tau,\cdot),(1-\Delta)^{\delta/2}\phi)=&(G,(1-\Delta)^{\delta/2}\phi)+
            \int_{\tau}^{T}(F(s,\cdot),(1-\Delta)^{\delta/2}\phi)\, ds\\
&-\int_{\tau}^{T}(v^{l}(s,\cdot),(1-\Delta)^{\delta/2}\phi)\,
dW^{l}_{s}, \quad \phi\in H_{p/(p-1)}^{-k}.
  \end{split}
\end{eqnarray*}
Hence, for any $\phi\in H_{p/(p-1)}^{-k-\delta}$ the equality
$$(u(\tau,\cdot),\phi)=(G,\phi)+
            \int_{\tau}^{T}(F(s,\cdot),\phi)\, ds
                    -\int_{\tau}^{T}(v^{l}(s,\cdot),\phi)\, dW^{l}_{s}$$
 holds for all $\tau\in[0,T]$ with probability 1. Then $(u,v)$ solves
 BSDE \eqref{bsde} for $n=k+\delta$ in the sense of Definition \ref{S bsde},
 and satisfies the inequality~(\ref{Estimate of bsde}) with $n:=k+\delta$ which is
  exactly the inequality~(\ref{InductionEst}).

In what follows, we shall use the method of finite-dimensional
approximation.

Since $\bH_{p,\infty}^n \subset \bH_p^n$ and $\|u\|_{\bH_p^n}\leq C(T,p)
\|u\|_{\bH_{p,\infty}^n}$ for $u\in \bH_{p,\infty}^n,$  it remains to prove
the existence of the solution $(u,v)$ in $
\bH_{p,\infty}^n\times\bH_{p,2}^n,$ the assertion (ii) and the following
estimate
    \begin{equation}\label{Estimate of bsde 2}
      \|u\|_{\bH_{p,\infty}^{n}} +\|v\|_{\bH_{p,2}^{n}}\leq
      c(p,T)\left(\|F\|_{\bH^{n}_{p}}+\|G\|_{L^{p}(\Omega,\sF_{T},H_{p}^{n})}
      \right).
    \end{equation}

It is known (see~\cite{HandbookofBanach}) that the Banach space
$L^{p}(\bR^{d})$ has a Schauder basis for $p\in (1,\infty)$. Let
$\{e_{i}:i=1,2,3,\dots\}$ be a Schauder basis of $L^{p}(\bR^{d}).$
Then there exists an $M\in (0,\infty)$ and a unique sequence bounded
linear functional $a_i$ such that for any $h\in L^{p}(\bR^{d}),$ we
have
$$\sup_{j\geq 1}\left\|\sum_{i=1}^{j}a_i(h)e_i\right\|_{p}\leq M
\|h\|_p \quad \textrm{and}\quad \lim_{j \to
\infty}\left\|h-\sum_{i=1}^{j}a_i(h)e_i\right\|_p=0.$$ In
particular, for convenient discussion, we consider $e_{i}(x)$ to be
finite for every $x\in\bR^d$ and $i=1,2,3,\dots.$

 By \cite{Hu_2002}, there exist  uniquely
$U_{k}:=(U_{k1},\dots,U_{kk})^{\mathcal {T}}\in \cS^p(\bR^k)$ and a
$\sP$-measurable process $V^{l}:=(V_{ki}^l)\in
L^{p}(\Omega,L^2([0,T],\bR^{k\bigotimes m}))$ which solve the scalar valued
BSDE
\begin{equation}\label{finite dim 1}
  \left\{\begin{array}{l}
    \begin{split}
    -dU_{ki}=F_{ki}dt-V_{ki}^{l}dW^{l}_t,
    ~~t\in [0,T],
    \end{split}\\
    \begin{split}
      U_{ki}(T)=G_{ki},
    \end{split}
  \end{array}\right.
\end{equation}
where $G_{ki}=a_i(G)$ and $F_{ki}=a_i(F(t,\cdot)),$ with $i=1,2,\dots,k.$
Denote $\vec{G}_k:=(G_{k1},\dots,G_{kk})^{\mathcal {T}}$ and
$\vec{F}_k:=(F_{k1},\dots,F_{kk})^{\mathcal {T}}.$
 We have
\begin{eqnarray}\label{proof of bsde 1}
    \begin{split}
    &E\left[\sup_{t\leq T} |U_{k}(t)|^{p}\right]
    +E\left[\left(\int_{0}^{T}|V_{k}(t)|^{2}dt\right)^{p/2}\right]\\
    \leq &c(T,k,p)
    \left\{E\left[|\vec{G}_{k}|^{p}\right]
    +E\left[\int_{0}^{T}|\vec{F}_{k}(t)|^{p}dt        \right]\right\}
    \end{split}
\end{eqnarray}
and with probability 1
\begin{eqnarray}\label{proof of bsde 11}
    U_k(t)=\vec{G}_k+\int_t^T \vec{F}_k(s)ds-\int_t^TV_k^l(s)dW^l_s,~t\in
    [0,T].
\end{eqnarray}
 Define
\begin{eqnarray}\label{A}
    u_{k}:=\sum_{i=1}^{k}U_{ki}e_{i}, \quad
    v_{k}:=\sum_{i=1}^{k}V_{ki}e_{i},\quad  G_k:=\sum_{i=1}^{k}G_{ki}e_{i},\quad \textrm{ and } F_k:=
\sum_{i=1}^{k}F_{ki}e_{i}.
\end{eqnarray} It is obvious that $u_k,v_k,$ and $F_k$ are all $\sP\times\cB (\bR^d)$-measurable
processes. In view of \eqref{proof of bsde 1}, we can check that
 the pair $(u_{k},v_{k})$ solves the Banach space-valued BSDE \eqref{bsde}
  with $(F,G):=(F_{k},G_{k})$ in the sense of Definition \ref{S bsde}.
Moreover, for any $x\in\bR^d,$ the pair $(u_k(\cdot,x),v(\cdot,x))$ solves
the scalar valued BSDE
\begin{equation}\label{1003_1}
  \left\{\begin{array}{l}
    \begin{split}
    -du_{k}(t,x)=F_{k}(t,x)dt-v^l_{k}(t,x)dW^{l}_t,
    ~~t\in [0,T],
    \end{split}\\
    \begin{split}
      u_{k}(T,x)=G_{k}(x),
    \end{split}
  \end{array}\right.
\end{equation}
and satisfies the following estimate
\begin{eqnarray}\label{1003_2}
    \begin{split}
    &E\big[\sup_{t\leq T} |u_{k}(t,x)|^{p}\big]
    +E\left[\int_{0}^{T}|v_{k}(t,x)|^{2}dt\right]^{\frac{p}{2}}\\
    \leq\,& C  \bigg\{E\big[|G_{k}(x)|^{p}\big]
    +E\left[\int_{0}^{T}|F_{k}(t,x)|^{p}dt        \right]\bigg\}
    \end{split}
\end{eqnarray}
where $C=C(T,p) $ does not depend on $k$ since the constant in the
BDG inequality is universal and does not depend on the dimension of
the range space of the underlying local martingale. Integrating both
sides of the last inequality on $\bR^{d}$ and then applying the
Fubini theorem, we get the pair
$(u_{k},v_{k})\in\bH^{0}_{p,\infty}\times \bH^{0}_{p,2}$ satisfies
the following inequality
\begin{eqnarray}\label{proof of bsde 2}
      \|u_{k}\|_{\bH_{p,\infty}^{0}}+\|v_{k}\|_{\bH_{p,2}^{0}}\leq
      c(p,T)\{\|F_{k}\|_{\bH^{0}_{p}}+\|G_{k}\|_{L^{p}(\Omega,\sF_{T},H_{p}^{0})}     \}.
\end{eqnarray}

On the other hand, as $\|F_k(\omega,t)-F(\omega,t)\|_p\rightarrow 0$ and
$\|F_k(\omega,t)-F(\omega,t)\|\leq(M+1)\|F(\omega,t)\|_p$ for
$(\omega,t)\in \Omega\times [0,T],a.e.,$ by using the dominated convergence
theorem we have $F_{k}\rightarrow F$ strongly in $\bH_p^0 $ as $k
\rightarrow\infty.$ Similarly, $G_{k}\rightarrow G$ strongly in
$L^{p}(\Omega,\sF_{T},H_{p}^{0})$ as $k \rightarrow\infty.$
Hence, there exists
  $(u,v)\in\bH_{p,\infty}^{0}\times \bH_{p,2}^{0}$ such
that it is the strong limit of the sequence $\{(u_{k},v_{k})\}$ in
 $\bH_{p,\infty}^{0} \times \bH_{p,2}^{0}\textrm{ as } k
\rightarrow\infty,$ and satisfies the estimate \eqref{Estimate of bsde 2}.

Furthermore, in view of \eqref{proof of bsde 11} and \eqref{A}, we conclude
that, for any $\phi\in L^{p/(p-1)}(\bR^d)$ the equality
\begin{eqnarray}\label{proof of bsde 3}
    (u_{k}(\tau,\cdot),\phi)=(G_{k},\phi)+
            \int_{\tau}^{T}(F_{k}(s,\cdot),\phi)ds
                    -\int_{\tau}^{T}(v_{k}^{l}(s,\cdot),\phi)dW^{l}_{s}
\end{eqnarray}
holds for all $\tau\in [0,T)$  with probability 1. Since
\begin{eqnarray*}
    \begin{split}
    &E\left[\int_0^T|\int_{\tau}^{T}(v_{k}^{l}-v^l(s,\cdot),\phi)dW^{l}_{s}|d\tau
    \right] \\
        \leq& C
            E\left[\int_0^T \left(\int_{\tau}^T|(v_{k}-v(s,\cdot),\phi)|^2
            ds\right)^{1/2} d\tau \right]              \\
        \leq& C(T)
            E\left[\left(\int_{0}^T\left(\int_{\bR^d}
            |(v_k-v)(s,x)\phi(x)|dx\right)^2ds\right)^{1/2}\right]   \\
        & \textrm{(using Minkowski inequality)}\\
        \leq& C(T)
            E\left[\int_{\bR^d}\left(\int_{0}^T\
            |(v_k-v)(s,x)\phi(x)|^2ds\right)^{1/2}dx\right]   \\
        \leq& C(T)
         \|v_k-v\|_{\bH_{p,2}^0}\|\phi\|_{p/(p-1)} \rightarrow 0
        \textrm{ as }k\rightarrow \infty,           \\
    \end{split}
\end{eqnarray*}
\begin{eqnarray*}
    \begin{split}
    E\Big[\int_0^T |(u_{k}(\tau,\cdot)-u(\tau,\cdot),\phi)| d\tau\Big]&=
        E\left[\int_{0}^T\int_{\bR^d}|(u_k-u)(s,x)
        \phi(x)|dxds\right]   \\
        &\leq T \|u_k-u\|_{\bH_{p,\infty}^0}\|\phi\|_{p/(p-1)} \rightarrow 0
        \textrm{ as }k\rightarrow \infty,
    \end{split}
\end{eqnarray*}
and
$$E\Big[\int_0^T|(G_k-G,\phi)|ds
\Big]+E\Big[\int_0^T|\int_{\tau}^T((F_k-F)(s,\cdot),\phi)ds|d\tau \Big]
\rightarrow 0 \textrm{ as }k\rightarrow \infty.$$
 Taking limits in $L^1(\Omega\times
[0,T],\sF_T\times \cB(\bR^d)),$  on both sides of the equation \eqref{proof
of bsde 3} we conclude \eqref{S bsde2} almost everywhere in $[0,T]\times
\Omega$.

Since, for any $\phi\in L^{p/(p-1)}(\bR^d),$ Equation \eqref{proof
of bsde 3} holds for all $\tau\leq T$ with probability 1, the
process $\{(u_k(t,\cdot),\phi), t\in [0,T]\}$ is continuous
    $(a.s).$ As
$$E \sup_{0\leq t \leq T}\|u-u_k\|_{L^p(\bR^d)} \leq
        \|u-u_k\|_{\bH_{p,\infty}^{0}}
         \rightarrow 0,
    $$ the process $\{(u(t,\cdot),\phi), t\in [0,T]\}$ is continuous.
This implies that, for any $\phi\in L^{p/(p-1)}(\bR^d),$ equality \eqref{S
bsde2} holds not only in $[0,T]\times \Omega$ almost everywhere but also
 for all $\tau\leq T$ almost surely.

Besides, since $u_k\in C([0,T],L^p(\bR^d))(a.s.)$ and $E \sup_{0\leq t \leq
T}\|u-u_k\|_{L^p(\bR^d)}
         \rightarrow 0~as~k\rightarrow \infty,$ we have $u\in C([0,T],L^p(\bR^d))(a.s.).$
%
We complete the proof of the
lemma.
\end{proof}
%
\begin{rmk}
In view of Lemma \ref{lem finite approximation H_pn}, we can
approximate in $\bH_p^0\times L^{p}(\Omega,\sF_{T},H_{p}^{0})$ for
$p\in (1,2]$ during the proof $(F,G)$ by a sequence $(F_k,G_k)$
belonging to $\bH_2^0\times L^{2}(\Omega,\sF_{T},H_{2}^{0})$.
Moreover, we can assume that $(F_k,G_k)(\omega,t)$ is uniformly
compactly supported in $\bR^d$ for $(\omega,t)\in\Omega\times [0,T]$
$a.e.$. After finite-dimension approximation of $(F_k,G_k)$ in
$\bH_2^0\times L^{2}(\Omega,\sF_{T},H_{2}^{0})$ where a Hilbert
basis is a Shauder basis, the rest of our proof goes in a standard
way (c.f. \cite{Hu_2002}) for $p\in (1,2]$, while not for $p\in
(2,\infty).$
\end{rmk}

\begin{lem}\label{BSDE lem2}
    Let $(u,v)\in\bH_{p,\infty}^{n}\times \bH_{p,2}^{n}$ be a solution
    of \eqref{bsde} for given
    $F\in\bH^{n}_{p}$ and $G=0$. Then for any $\eps >0,$
    there exists a positive constant $c=c(p,T,\eps)<\infty$ such that
\begin{eqnarray}\label{BSDE lem2_0}
  \begin{split}
    \|v\|_{\bH_{p,2}^{n}(t)}\leq
            c(p,T,\eps)\|u\|_{\bH_{p}^{n}(t)}+
                \eps \|F\|_{\bH^{n}_{p}(t)},~t\in [0,T].
  \end{split}
\end{eqnarray}
\end{lem}
\begin{rmk}
  Lemma \ref{BSDE lem2} yields that
   for any $A\in \sP \times \cB (\bR^d),$  there holds
\begin{eqnarray*}
  \begin{split}
   \|v\mathbb{I}_{A}\|_{\bH_{p,2}^{n}(t)}\leq
            c(p,T,\eps)\|u\mathbb{I}_{A}\|_{\bH_{p}^{n}(t)}+
                \eps \|F\mathbb{I}_{A}\|_{\bH^{n}_{p}(t)}.
  \end{split}
\end{eqnarray*}
    In particular, if $A:=\{(t,\omega,x)\in [0,T]\times \Omega
\times \bR^d:\, u(t,\omega,x)=0 \},$ and as $\eps$ is arbitrary, then we
get $\|v\mathbb{I}_{A}\|_{\bH_{p,2}^{n}(t,T)}=0$ which implies
$$v\mathbb{I}_{\{u=0 \}}= 0\textrm{ for }(\omega,t,x)\in \Omega\times [0,T]\times \bR^d,a.e..$$
\end{rmk}

\begin{proof}[Proof of Lemma \ref{BSDE lem2}]
    First consider the case of $n=0$.
    Without loss of generality, we assume that the Brownian motion
     is one-dimensional.

Consider the approximation sequence $\{(u_k,v_k)\}$ defined in the proof of
Lemma \ref{lem BSDE}. For any fixed $x\in \bR^d$ the pair $(u_k,v_k)$
solves the following scalar valued BSDE
\begin{equation*}
  \left\{\begin{array}{l}
    \begin{split}
    -du_{k}(t,x)=F_{k}(t,x)\, dt-v_{k}(t,x)\, dW_t,
    ~~t\in [0,T],
    \end{split}\\
    \begin{split}
      u_{k}(T,x)=0,
    \end{split}
  \end{array}\right.
\end{equation*}
and satisfies the following inequality (see \eqref{1003_1} and
\eqref{1003_2}).
\begin{eqnarray*}
    \begin{split}
    E\left[\sup_{t\leq T} |u_{k}(t,x)|^{p}\right]
    +E\left[\int_{0}^{T}|v_{k}(t,x)|^{2}\, dt\right]^{\frac{p}{2}}
    \leq C(p,T)E\left[\int_{0}^{T}|F_{k}(t,x)|^{p}\, dt        \right].
    \end{split}
\end{eqnarray*}

 For each integer $l\geq 1,$ define the stopping time
$$  \tau _{l}:=\textrm{inf}\{
    t\in [0,T],\int_0^t|v(r,x)|^2dr \geq l\}\wedge T.
  $$
Using It\^o's formula, we have
\begin{eqnarray*}
  \begin{split}
|u_k(\eta,x)|^2+\int_{\eta}^{\tau_{l}} |v_k(r,x)|^2dr
    =&|u_k(\tau_l,x)|^2+2\int_{\eta}^{\tau_l}u_k(r,x)F_k(r,x)dr     \\
      -2\int_{\eta}^{\tau_l}u_k(r,x)v_k(r,x)dW_r,a.s.
  \end{split}
\end{eqnarray*}
for any stopping time $\eta\leq \tau_l$. Therefore,
\begin{eqnarray*}
  \begin{split}
&\left(\int_{\eta}^{\tau_{l}} |v_k(r,x)|^2dr\right)^{p/2}       \\
 \leq&
    c(p)\left(\sup_{t\in [\eta,T]}|u_k(t,x)|^{p}
        + \left[\int_{\eta}^{T}|u_k(r,x)F_k(r,x)|dr\right]^{p/2}
        +\left|\int_{\eta}^{\tau_l}u_k(r,x)v_k(r,x)dW_r\right|^{p/2}
        \right).
  \end{split}
\end{eqnarray*}
Noting by the BDG inequality that
\begin{eqnarray*}
  \begin{split}
    E\left[\left|\int_{\eta}^{\tau_l}u_k(r,x)v_k(r,x)dW_r\right|^{p/2}\right]
    &\leq    c_1(p)E\left[\left(
                    \int_{\eta}^{\tau_l}|u_k(r,x)v_k(r,x)|^2dr
                                                \right)^{p/4}\right]  \\
    &\leq    c_1(p)E\left[ \sup_{t\in [\eta,T]}|u_k(t,x)|^{p/2}
                        \left(\int_{\eta}^T|v_k(r,x)|^2dr\right)^{p/4}
                  \right],
  \end{split}
\end{eqnarray*}
we have
\begin{eqnarray*}
  \begin{split}
    &E\left[\left(\int_{\eta}^{\tau_{l}} |v_k(r,x)|^2dr\right)^{p/2}\right]     \\
     \leq& c(p)E\left[\sup_{t\in [\eta,T]} |u_k(t,x)|^p+
      \left(\int_{\eta}^{T}|u_k(r,x)F_k(r,x)|dr\right)^{p/2}
        +|\int_{\eta}^{\tau_l}u_k(r,x)v_k(r,x)dW_r|^{p/2} \right]             \\
     \leq& c(p) E\left[\sup_{t\in [\eta,T]} |u_k(t,x)|^p+
      \left(\int_{\eta}^{T}|u_k(r,x)F_k(r,x)|dr\right)^{p/2}
                \right]                                                        \\
       &   ~~~~~~~~~      +c_1(p)E\left[ \sup_{t\in [\eta,T]}|u_k(t,x)|^{p/2}
                        \left(\int_{\eta}^T|v_k(r,x)|^2dr\right)^{p/4}
                  \right]                                                      \\
     \leq& c(p) E\left[ \sup_{t\in [\eta,T]} |u_k(t,x)|^p
    +\left(\int_{\eta}^{T}|u_k(r,x)F_k(r,x)|dr\right)^{p/2}\right]\\
                &+\frac{1}{2} E\left[\left(\int_{\eta}^T|v_k(r,x)|^2dr\right)^{p/2}
                \right],
  \end{split}
\end{eqnarray*}
and, for each $l\geq 1$ and $\forall \eps_1 >0,$ there is
$c=c(p,\eps_1,T)>0$ such that
\begin{eqnarray*}
  \begin{split}
    &E\left[\left(\int_{\eta}^{\tau_{l}} |v_k(r,x)|^2dr\right)^{p/2}\right]     \\
     &\leq c(p) E\left[ \sup_{t\in [\eta,T]} |u_k(t,x)|^p
                +\left(\int_{\eta}^{T}|u_k(r,x)F_k(r,x)|dr\right)^{p/2}\right]  \\
     &\leq c(p) E\left[\sup_{t\in [\eta,T]} |u_k(t,x)|^p
                +\sup_{t\in [\eta,T]} |u_k(t,x)|^{p/2}
                \left(\int_{\eta}^{T}|F_k(r,x)|dr\right)^{p/2}\right]           \\
     &\leq c(p,\eps_1,T) E\big[\sup_{t\in [\eta,T]} |u_k(t,x)|^p\big]
                    +\eps_1 E\left[\int_{\eta}^{T}|F_k(r,x)|^{p}dr
                    \right].
  \end{split}
\end{eqnarray*}
So, letting $l\rightarrow \infty$ and using Fatou's lemma, we have
\begin{eqnarray}\label{BSDE lem2_1}
    \begin{split}
    &~~~~~ E\left[\left(\int_{\eta}^{T} |v_k(r,x)|^2dr\right)^{p/2}\right]           \\
      &\leq c(p,\eps_1,T) E\big[\sup_{t\in [\eta,T]} |u_k(t,x)|^p\big]
                    +\eps_1 E\left[\int_{\eta}^{T}|F_k(r,x)|^{p}dr
                    \right]
    \end{split}
\end{eqnarray}
for any stopping time $\eta\in[0,T],$ and in particular for any
deterministic $\eta \in [0,T].$

On the other hand, using Corollary 2.3 of Briand et al. \cite{Hu_2002},
 we have almost surely
\begin{eqnarray}\label{BSDE lem2_2}
  \begin{split}
    &~~~~~~|u_k(t,x)|^p+c_0(p)\int_t^T
            |u_k(s,x)|^{p-2}\mathbb{I}_{\{u_k(s,x)\neq 0\}} |v_k(s,x)|^2\, ds   \\
    &\leq \int_t^T|u_k(s,x)|^{p-1}|F_k(s,x)|\, ds
            -p\int_t^T|u_k(s,x)|^{p-2}u_k(s,x)v_k(s,x)\, dW_s,~t\in [0,T]
  \end{split}
\end{eqnarray}
where $c_0(p)=p[(p-1)\wedge 1]/2.$

As $(u_k,v_k)\in\bH_{p,\infty}^{0}\times \bH_{p,2}^{0},$ from the preceding
inequality, we have almost surely
$$\int_t^T |u_k(s,x)|^{p-2}\mathbb{I}_{\{u_k(s,x)\neq 0\}} |v_k(s,x)|^2ds < \infty,
~ t\in [0,T], $$ and further,
\begin{eqnarray}\label{BSDE lem2_3}
\begin{split}
    &c_0(p)E\left[
    \int_t^T |u_k(s,x)|^{p-2}\mathbb{I}_{\{u_k(s,x)\neq 0\}} |v_k(s,x)|^2ds
        \right]        \\
     \leq& E\left[ \int_t^T|u_k(s,x)|^{p-1}|F_k(s,x)|ds  \right]
     ,~~~~~~~t\in [0,T].
\end{split}
\end{eqnarray}
From \eqref{BSDE lem2_2} and \eqref{BSDE lem2_3}, using the BDG inequality
we have
\begin{eqnarray*}
  \begin{split}
    &~~~~E\big[ \sup_{s\in [t,T]} |u_k(s,x)|^p \big]         \\
     & \leq
     E\left[ \int_t^T|u_k(s,x)|^{p-1}|F_k(s,x)|ds  \right]
        +E\left[ |\int_t^T|u_k(s,x)|^{p-2}u_k(s,x)v_k(s,x)dW_s| \right]\\
     &~~~~~~
        +E\left[ \sup_{r\in [t,T]}
        |\int_t^r|u_k(s,x)|^{p-2}u_k(s,x)v_k(s,x)dW_s|    \right]      \\
     & \leq
        E\left[ \int_t^T|u_k(s,x)|^{p-1}|F_k(s,x)|ds  \right]
          +c(p)E\left[ \left(\int_t^T
          (|u_k(s,x)|^{p-1}|v_k(s,x)|)^2ds\right)^{1/2}
            \right]                                             \\
     & \leq
        E\left[ \int_t^T|u_k(s,x)|^{p-1}|F_k(s,x)|ds  \right]   \\
     &~~~~~~ +c(p)E\left[\sup_{s\in [t,T]}|u_k(s,x)|^{p/2}
          \left(
          \int_t^T|u_k(s,x)|^{p-2}\mathbb{I}_{\{u_k(s,x)\neq 0\}} |v_k(s,x)|^2ds
          \right)^{1/2}\right]                                   \\
     & \leq
        E\left[ \int_t^T|u_k(s,x)|^{p-1}|F_k(s,x)|ds  \right]
        +\frac{1}{2}E\big[ \sup_{s\in [t,T]} |u_k(s,x)|^p \big] \\
     &~~~~~+\frac{c(p)^2}{2}E\left[
        \int_t^T|u_k(s,x)|^{p-2}\mathbb{I}_{\{u_k(s,x)\neq 0\}} |v_k(s,x)|^2ds
        \right]                                                  \\
     &\leq
        c'(p)E\left[ \int_t^T|u_k(s,x)|^{p-1}|F_k(s,x)|ds  \right]
        +\frac{1}{2}E\big[ \sup_{s\in [t,T]} |u_k(s,x)|^p \big].
  \end{split}
\end{eqnarray*}
Thus, for any $\eps_2>0,$ we have
\begin{eqnarray}
  \begin{split}
    &E\big[ \sup_{s\in [t,T]} |u_k(s,x)|^p \big] \leq
        2c'(p)E\left[ \int_t^T|u_k(s,x)|^{p-1}|F_k(s,x)|ds  \right]  \\
    &~~~~~~~~~~\leq
        c(p,\eps_2)E\left[\int_t^T|u_k(s,x)|^pds\right]+
                \eps_2 E\left[\int_t^T |F_k(s,x)|^p ds\right].
  \end{split}
\end{eqnarray}
Combining the lat inequality with \eqref{BSDE lem2_1}, and letting $\eps_1$
and $\eps_2$ be small enough such that $\eps_2  c(p,\eps_1,T)
+\eps_1<\eps$, we get
\begin{eqnarray}
  \begin{split}
     E\left[\left(\int_{\eta}^{T} |v_k(r,x)|^2dr\right)^{p/2}\right]
         \leq   C E\left[\int_t^T|u_k(s,x)|^pds\right]+
                \eps E\left[\int_t^T |F_k(s,x)|^p ds\right].
  \end{split}
\end{eqnarray}
Here the constant $C=C(p,T,\eps)$ is independent of $k.$

Now, integrating on $\bR^d$ both sides of the preceding inequality
and letting $k\rightarrow \infty $ , we get \eqref{BSDE lem2_0}
for $n=0.$ The general case can be proved by induction. The proof is
complete.
\end{proof}

\begin{rmk}
  The arguments in the proof of Lemma~\ref{BSDE lem2} are more or less standard (see pages 115--118 of \cite{Hu_2002}).
\end{rmk}

\section {A stochastic Banach Space}
In this section we shall define a stochastic Banach space which will play a
crucial role in $L^p$ theory of BSPDEs.

\begin{defn}\label{stochastic space}
  For $n\in\bR,$ $p\in(1,\infty)$ and a $\sD'$-valued function
  $u\in\bH_{p}^{n},$
  we say $u\in\mathscr{H}_{p}^{n}$ if
    $u_{xx}\in\bH_{p}^{n-2},u(T,\cdot)\in L^{p}(\Omega,\sF_{T},B_{p,p}^{n-2/p}),$
     and there exists
  $(F,v)\in \bH_{p}^{n-2}\times \bH_{p}^{n-1}$
   such that, $\forall \phi \in C^{\infty}_{c},$ the following equality
  \begin{eqnarray}\label{stochastic dfn 1}
    (u(t,\cdot),\phi)
            =(u(T,\cdot),\phi)+\int_{t}^{T}(F(s,\cdot),\phi)ds
                   -\sum_{r=1}^{m}\int_{t}^{T}(v^{k}(s,\cdot),\phi)dW^{k}_{s},
  \end{eqnarray}
  holds for all $t\leq T$ with probability 1.

  Define
   $\sH_{p,0}^{n}:=\sH_{p}^{n}\cap \{u:u(T,\cdot)=0\},$
   and for $u\in\mathscr{H}_{p}^{n}$
    \begin{eqnarray}\label{stochastic norm1}
        \|u\|_{\sH_{p}^{n}}:=  \|u_{xx}\|_{\bH_{p}^{n-2}}
                +\|F\|_{\bH_{p}^{n-2}}+    \|v_x\|_{\bH_{p}^{n-2}}
                         +\left(E\|u(T,\cdot)\|_{B_{p,p}^{n-2/p}}^{p}\right)^{\frac{1}{p}}.
    \end{eqnarray}
\end{defn}

\begin{rmk}\label{rmk 2}
 Note that $L^{p}(\Omega,\sF_{T},B_{p,p}^{n-2/p}) $ is
  continuously embedded into
$   L^{p}(\Omega,\sF_{T},H_{p}^{n-2}).$ If $u\in\mathscr{H}_{p}^{n},$ it
follows from Lemma \ref{lem BSDE} that $u\in \bH_{p,\infty}^{n-2} ,v\in
\bH_{p,2}^{n-2},$ and
\begin{equation*}
    \begin{split}
    \left(E\sup_{t\leq T}\|u(t,\cdot)\|^p_{H_{p}^{n-2}}\right)^{1/p}+
            \|v\|_{\bH_{p,2}^{n-2}}  &\leq
      \|u\|_{\bH_{p,\infty}^{n-2}}+\|v\|_{\bH_{p,2}^{n-2}}\\
      &\leq  c(p,T)\left(\|F\|_{\bH^{n-2}_{p}}
                    +\|u(T,\cdot)\|_{L^{p}(\Omega,\sF_{T},H_{p}^{n-2})} \right)\\
        &\leq c(p,T) \left(\|F\|_{\bH^{n-2}_{p}}
                    +\|u(T,\cdot)\|_{L^{p}(\Omega,\sF_{T},B_{p,p}^{n-2/p})}  \right).
    \end{split}
\end{equation*}
\end{rmk}

\begin{rmk}\label{rmk 1}
  From Remarks \ref{rmk 2} and \ref{rmk of bsde}, the fact that $u\in \sH_p^n$
   implies, in some sense $\{u(t,x)\}_{0\leq t\leq T}$ is a
semi-martingale of drift
    $F(t,x)_{0\leq t\leq T}$ and diffusion $v(t,x)_{0\leq t\leq T}.$ Further, by
  Lemma \ref{lem finite approximation H_pn} and the estimates in
  Remark \ref{rmk 2}, Doob-Meyer decomposition theorem implies
  the uniqueness of $(F,v)$. Therefore, the norm \eqref{stochastic norm1} is well defined.
   Without confusions,
   we shall always say that $F$ and $v$ are the drift term and diffusion term of
   $u$ respectively. In the following, we denote the diffusion term $v$ of $u$
    by $\mathbb{D}u.$

   On the other hand, it is worth noting that the
   elements  of $\mathscr{H}_{p}^{n}$ are assumed to be defined for all
   $(\omega,t)$ and take values in $\sD',$ and that
   $\mathscr{H}_{p}^{n}$ is a normed linear space in which we identify
   two elements $u_1$ and $u_2$ if $\|u_1 -u_2\|_{\mathscr{H}_{p}^{n}}=0.$
   In view of Definition \ref{stochastic space}, for any $p,q\in(1,\infty)$ and
   $n,r\in\bR,$ if $u\in\sH_p^n$ and
   $\|u\|_{\sH_q^r}<\infty,$ one can check that $u\in\sH_q^r$ and that,
    in particular, $\|u\|_{\sH_p^n}=0$ implies $\|u\|_{\sH_q^r}=0.$
\end{rmk}

\begin{thm}\label{stochastic banach property thm}
    The spaces $\mathscr{H}_{p}^{n}$ and $\sH_{p,0}^{n}$ equipped with norm
    \eqref{stochastic norm1} are Banach spaces. Moreover, we have
\begin{eqnarray}\label{thm1}
  \|u\|_{\bH_{p}^{n}} \leq
                    C(p,T)\|u\|_{\sH_{p}^{n}},\quad
  E\left[\sup_{t\leq T}\|u(t,\cdot)\|^p_{H_{p}^{n-2}}\right] \leq
                    C(p,T)\|u\|^p_{\sH_p^n }.
\end{eqnarray}
\end{thm}
\begin{proof} The second inequality of (\ref{thm1}) is given in Remark \ref{rmk
2}. Since
\begin{eqnarray*}
  \begin{split}
\|u\|_{\bH_{p}^{n}}=\|(1-\Delta)u\|_{\bH_{p}^{n-2}}
 &\leq  \|u\|_{\bH_{p}^{n-2}}+\|u\|_{\sH_{p}^{n}}          \\
&\leq T^{1/p}\left( E[\sup_{t\leq
            T}\|u(t,\cdot)\|^p_{H_{p}^{n-2}}]\right)^{1/p}
            +\|u\|_{\sH_{p}^{n}},
  \end{split}
\end{eqnarray*}
we have the first inequality of (\ref{thm1}).

It remains for us to show the completeness of $\mathscr{H}_{p}^{n}$.
Let $\{u_j\}$ be a Cauchy sequence in $\mathscr{H}_{p}^{n}.$ Then it
is also a Cauchy sequence in $\bH_{p}^{n}$, and there exists
$u\in\bH_{p}^{n}$ such that
$$\lim_{j\to \infty}\|u-u_j\|_{\bH_{p}^{n}}= 0.$$
Furthermore, $\{ u_{jxx}\}$ is a Cauchy sequence in $\bH_{p}^{n-2}$
and
 $$\lim_{j\to \infty}\|u_{jxx}-u_{xx}\|_{\bH_{p}^{n-2}}=0. $$
For $u_j (T),F_j,$ and the corresponding  $u_j,$ there exist
 $u(T)\in L^{p}(\Omega,\sF_{T},B_{p,p}^{n-2/p})\subset
  L^{p}(\Omega,\sF_{T},H_{p}^{n-2})$
  and $F\in\bH_{p}^{n-2}$ such that
  $$\lim_{j\to \infty}\|u(T)-u_j
  (T)\|_{L^{p}(\Omega,\sF_{T},B_{p,p}^{n-2/p})}=
  0,\quad \lim_{j\to \infty}\|u(T)-u_j (T)\|_{L^{p}(\Omega,\sF_{T},H_{p}^{n-2})}=
  0,$$
  and
  $$
    \lim_{j\to \infty}\|F-F_j\|_{\bH_{p}^{n-2}}=0.
 $$
Let $v_j$ be the diffusion term of $u_j.$ Using the argument from Remark
\ref{rmk 2}, we conclude that there is $v\in \bH_{p}^{n-1}\cap
\bH_{p,2}^{n-2}$ such that
$$ \lim_{j\to \infty}\|v_x-(v_j)_x\|_{\bH_{p}^{n-2}}= 0\quad
\textrm{and}\quad
   \lim_{j\to \infty}\|v-v_j\|_{\bH_{p,2}^{n-2}}=0.
    $$
Since for any $\phi \in C^{\infty}_{c}$ the equality
  \begin{eqnarray}
    (u_j(t,\cdot),\phi)
            =(u_j(T,\cdot),\phi)+\int_{t}^{T}(F_j(s,\cdot),\phi)\, ds
                   -\sum_{r=1}^{m}\int_{t}^{T}(v_j^{k}(s,\cdot),\phi)\, dW^{k}_{s}
  \end{eqnarray}
  holds for all $t\leq T$ with probability 1, by taking on both sides limits in
  $L^1([0,T]\times\Omega,\sF_T\times\cB(\bR^d))$, we show that
   for any $\phi \in C^{\infty}_{c}$
equality \eqref{stochastic dfn 1} holds in $[0,T]\times\Omega$ almost
everywhere.

Furthermore, \eqref{thm1} implies that for $u$ (at least for a modification
of $u$), we have
$$\lim_{j\to \infty}E\sup_{t\leq T}\|u(t,\cdot)-u_j(t,\cdot)\|^p_{H_{p}^{n-2}}=0.
    $$
Since the processes $\{(u_j (t,\cdot),\phi), t\in [0,T]\},
j=1,2,dots$ are all continuous, it follows that
$\{(u(t,\cdot),\phi), t\in [0,T]\}$ is continuous. Therefore, for
any $\phi \in C^{\infty}_{c},$ equality \eqref{thm1} not only holds
in
 $[0,T]\times \Omega$ almost everywhere but also for all $\tau \leq T$
 almost surely. Hence, $u\in \sH_{p}^{n}$ and $u_j$ converges to $u$
 in $\sH_{p}^{n}.$ So, $\sH_{p}^{n}$ is a Banach space.

Similarly, we can check the completeness of $\sH_{p,0}^{n}$. The proof is
complete.
\end{proof}
\begin{rmk}\label{rmk sotchastic banach thm}
The estimate \eqref{thm1} can be verified
 for $u\mathbb{I}_{(t,T]},$ $t\in [0,T).$ Especially, we have
 $$E\sup_{s\in
(t,T]}\|u(s,\cdot)\|^p_{H_{p}^{n-2}} \leq
                    C(p,T)\|u\|^p_{\sH_p^n(t) } $$
with $\|u\|_{\sH_p^n(t) }:=\|u\mathbb{I}_{(t,T]}\|_{\sH_p^n}.$
\end{rmk}

Now, we show an embedding result about the stochastic Banach space
$\sH_p^n.$
\begin{prop}\label{thm Embedding}
    For $u\in \sH_p^n$ and $v=\bD u$, the following assertions hold:

  (i) If $\beta:=n-d/p> 0,$ then $(u,v)\in L^p((0,T],\sP,\mathcal
  {C}^{\beta}(\bR^d))\times L^p((0,T],\sP,\mathcal
  {C}^{\beta-1}(\bR^d))$ satisfying
  $$ E\left[ \int_0^T\|u(t,\cdot)\|_{\mathcal {C}^{\beta}(\bR^d)}^p\, dt\right]
        \leq C(n,d,p)\|u\|^p_{\bH_p^n}   \leq C(T,n,d,p)\|u\|^p_{\sH_p^n}, $$
  where $\mathcal {C}^{\beta}(\bR^d) $ is the Zygmund space which is different from the
   ordinary H$\ddot{o}$lder spaces $C^{\beta}(\bR^d)$ only if $\beta$ is an
  integer. In particular, if $p\in(1,2],$ we also have
  $$ E\left[ \int_0^T\|v(t,\cdot)\|_{\mathcal {C}^{\beta-1}(\bR^d)}^p\, dt\right]
        \leq C(T,n,d,p)\|u\|^p_{\sH_p^n}. $$

  (ii) If $n>l$ and $n-d/p= l-d/q,$ then
  $$E\left[\int_0^T \|u(t,\cdot)\|^p_{l,q}\, dt\right]
        \leq C(l,n,d,p)\|u\|_{\bH_p^n}^p\leq  C(T,l,n,d,p)\|u\|_{\sH_p^n}^p .
    $$ In particular, if $p\in (1,2],$ we also have
   $$E\left[\int_0^T \|v(t,\cdot)\|^p_{l-1,q}\, dt\right]
        \leq C(T,l,n,d,p)\|u\|_{\sH_p^n}^p.
    $$

  (iii) If $ q\geq p $ and $\theta\in (0,1),$ then for
    $$n\geq l-\frac{d}{q}+\frac{d}{p}+2(1-\theta), $$
    we have $u\in L^{p/\theta}((0,T],H_q^l)~(a.s.)$ and
    $$ E\left[ \Big( \int_0^T\|u(t,\cdot)\|^{p/\theta}_{l,q}\, dt \Big)^{\theta}\right]
        \leq C(T,n,l,q,d,p,\theta) \|u\|_{\sH_p^n}^p.
    $$
    In particular, if $$q>p \quad \textrm{and}\quad
    n\geq l+\frac{d}{p}+\frac{2q-2p-d}{q},$$
by taking $\theta=pq^{-1},$ we have
$$ E\left[ \Big( \int_0^T\|u(t,\cdot)\|^q_{l,q}\, dt \Big)^{p/q}\right]
        \leq C(T,n,l,q,d,p) \|u\|_{\sH_p^n}^p.
    $$
\end{prop}
\begin{proof}
  By Lemma \ref{lem BSDE} and Theorem \ref{stochastic banach property thm},
    the assertions (i) and (ii) are straightforward in view of the classical Sobolev
  embedding theorems, which say that under conditions in (i) and
  (ii), we have $H_p^n\subset\mathcal{C}^{\beta}(\bR^d)$ and $H_p^n\subset H_q^l,$
  respectively. On the other hand, from the Sobolev embedding
  theorems, we get
  $$\|f\|_{l,q}\leq C(l,d,q,p)\|f\|_{l+d/p-d/q,p}
        \leq C(l,d,q,p,\theta)\|f\|^{1-\theta}_{n'-2,p}
            \|f\|_{n',p}^{\theta}~~,
    $$
    where $n':=l+d/p-d/q+2(1-\theta)\leq n.$
    Hence,
\begin{eqnarray*}
  \begin{split}
      E\left[ \Big( \int_0^T\|u(t,\cdot)\|^{p/\theta}_{l,q}\, dt
      \Big)^{\theta}\right]
  \leq &C E\left[ \Big(\int_0^T\|u(t,\cdot)\|^{(1-\theta)p/\theta}_{n'-2,p}
    \|u(t,\cdot)\|_{n',p}^{p}\, dt \Big)^{\theta}\right] \\
  \leq &C E\left[ \Big(\int_0^T\|u(t,\cdot)\|^{(1-\theta)p/\theta}_{n-2,p}
    \|u(t,\cdot)\|_{n,p}^{p}\, dt \Big)^{\theta}\right] \\
  \leq &C E\left[ \sup_{t\leq T}\|u\|_{n-2,p}^{(1-\theta)p}
            \Big( \int_0^T\|u(t,\cdot)\|_{n,p}^p\, dt \Big)^{\theta} \right]\\
  \leq &C\left( E\left[\sup_{t\leq T}\|u\|_{n-2,p}^p\right]
             + \|u\|_{\bH_p^n}^p \right) \\
  \leq &C \|u\|_{\sH_p^n}^p.
  \end{split}
\end{eqnarray*}
The last inequality is derived from Theorem \ref{stochastic banach
property thm}, and $C=C(T,n,l,q,d,p,\theta).$ The proof is complete.
\end{proof}

\section{$L^p$ solution of BSPDEs}
\subsection{Assumptions and the notion of the solution to BSPDEs}
Let $B(\bR^d)$ be the Banach spaces of bounded continuous functions
on $\bR^d,$ $C^{|n|-1,1}(\bR^d)$ the Banach space of $|n|-1$ times
continuously differentiable functions with the $(|n|-1)$th
derivatives satisfying the Lipschitz condition on $\bR^d,$ and
$C^{|n|+\gamma}(\bR^d)$ the usual H$\ddot{\textrm{o}}$lder space.
 The space $B^{|n|+\gamma}$ of Krylov \cite{Krylov_99} is defined as follows.
\begin{displaymath}
    B^{|n|+\gamma}=
    \left \{ \begin{array}{ll}
     B(\bR^d) &\textrm{if $n=0,$}\\
     C^{|n|-1,1}(\bR^d)&\textrm{if $n=\pm 1,\pm 2,\dots,$} \\
     C^{|n|+\gamma}(\bR^d) &\textrm{otherwise.}
     \end{array} \right.
\end{displaymath}
 Here, $n \in (-\infty,\infty),$ and
$\gamma\in[0,1) $ is fixed such that $\gamma =0 $ if $n$ is an integer;
 $\gamma > 0$ otherwise is so small that $|n|+\gamma$ is not an integer.

Consider the following semi-linear BSPDE:
\begin{equation}\label{BSPDE}
  \left\{\begin{array}{l}
    \begin{split}
      -d u (t,x)=&\big[
      a^{ij}(t,x)u_{x^i x^j}(t,x) +\sigma^{ik}(t,x)v^{k}_{x^i}(t,x)
      +F(u,v,t,x)\big]\, dt\\
      &-v^l(t,x)dW^l_t,\quad (t,x)\in [0,T]\times \bR^{d};
    \end{split}\\
    \begin{split}
    u(T,x)=G(x),\quad x\in \bR^{d}.
    \end{split}
  \end{array}\right.
\end{equation}

Here and in the following, denote
$$u_{x^i x^j}:=\frac{\partial^2}{\partial x^i \partial x^j}u,
~ u_{x^i}:=\frac{\partial}{\partial x^i}u,
~v^k_{x^i}:=\frac{\partial}{\partial x^i}v^k,
    $$
$$ u_x:=\nabla u=(u_{x^1},\dots,u_{x^d}),
~u_{xx}:=(u_{x^ix^j})_{1\leq i,j \leq d},
    $$
and
$$\alpha ^{ij}:= \frac{1}{2} \sum_{k=1}^{m}\sigma ^{ik} \sigma^{jk}.$$

\begin{ass} ($super\textrm{-}parabolicity$) \label{ass1}
     There exists a positive constant $\lambda$ such
     that
\begin{equation}
    [a^{ij}(t,x)-\alpha^{ij}(t,x)] \xi^i\xi^j
    \geq \lambda|\xi|^2
\end{equation}
    holds almost surely for all $x,\xi\in\bR^d$ and $t\in [0,T].$
\end{ass}

%
\begin{ass}   \label{ass2}
    There exists an increasing function $\kappa:
    [0,\infty)\rrow[0,\infty)$ such that $k(s)\downarrow 0$ as $s\downarrow0$
    and
    \begin{equation}
     \sum_{i,j=1}^d|a^{ij}(t,x)-a^{ij}(t,y)|
     +\sum_{i=1}^d\sum_{k=1}^m|\sigma^{ik}(t,x)-\sigma^{ik}(t,y)|\leq
            \kappa(|x-y|)
    \end{equation}
    holds almost surely for all $(t,x,y)\in [0,T]\times \bR^d\times \bR^d.$
\end{ass}


\begin{ass}     \label{ass3}
     The functions $a^{ij}(t,x)$ and $\sigma^{ik}(t,x)$ are real-valued
     $\sP\times\cB (\bR^{d})$-measurable, such that
\begin{equation}
  \begin{split}
    a^{ij}(t,\cdot),\sigma^{ik}(t,\cdot)\in B^{|n|+\gamma}, ~\textrm{and}~
    \|a^{ij}(t,\cdot)\|_{B^{|n|+\gamma}}
    +\|\sigma^{ik} (t,\cdot)\|_{B^{|n|+\gamma}}\leq \Lambda,
  \end{split}
\end{equation}
 almost surely for $i,j=1,\dots,d,k=1,\dots,m,$ and $t\in [0,T].$
\end{ass}

\begin{ass}\label{ass4}
$F(0,0,\cdot,\cdot)\in\bH_{p}^{n}.$ For $(u,v)\in H_{p}^{n+2}\times
H_{p}^{n+1},$ $F(u,v,t,\cdot)$ is an $H_{p}^{n}$-valued $\sP$-measurable
process such that there is a continuous and decreasing function
$\varrho:(0,\infty)\rrow[0,\infty)$ such that
 for any $\eps
>0,$   we have
\begin{equation}
  \begin{split}
  &\|F(u_1,v_1,t,\cdot)-F(u_2,v_2,t,\cdot)\|_{n,p}  \\
  \leq &\eps (\|u_1-u_2\|_{n+2,p}+\|v_1-v_2\|_{n+1,p})+
        \varrho (\eps)(\|u_1-u_2\|_{n,p}+\|v_1-v_2\|_{n,p} ),\\
         &\quad ~~u_{1},u_2\in H_{p}^{n+2} \textrm{ and } v_{1},v_2\in H_{p}^{n+1},
  \end{split}
\end{equation}
 holds for any $(t,\omega)\in [0,T]\times \Omega.$
\end{ass}
\begin{rmk}
    Assumption \ref{ass4} implies that $F(u,v,t,x)$ is Lipchitz
  continuous with respect to $(u,v)\in H_{p}^{n+2} \times H_{p}^{n+1}$ in
  $H_{p}^{n}$ for any $(t,\omega)\in(0,T]\times \Omega$, that is
  there is $ C>0$ such that
\begin{equation*}
  \begin{split}
 &\|F(u_1,v_1,t,\cdot)-F(u_2,v_2,t,\cdot)\|_{n,p} \\
 \leq&
        C (\|u_1-u_2\|_{n+2,p}+\|v_1-v_2\|_{n+1,p}), u_{1},u_2\in H_{p}^{n+2}
        \textrm{ and }  v_{1},v_2\in H_{p}^{n+1}.
  \end{split}
\end{equation*}
It also implies that $F$ does not depend on $u$ and $v$ if $\varrho
\equiv 0.$
\end{rmk}

\begin{defn}\label{definition of solution}
   We call $ u\in \sH_{p}^{n+2}$ a solution of BSPDE \eqref{BSPDE} if
   for any $\phi \in C^{\infty}_{c},$ the equality
\begin{equation}\label{def of solution}
  \begin{split}
  &(u(\tau,\cdot),\phi)=(G,\phi)+\int_{\tau}^{T}( a^{ij}(t,\cdot)u_{x^i x^j}(t,\cdot) +
       \sigma^{ik}(t,\cdot)(\bD u)^{k}_{x^i}(t,\cdot)
          +F(u,\bD u,t,\cdot),\phi)\, dt\\
  &~~~~~~~~~~~~~~~~~~~~~~~-\int_{\tau}^{T}((\bD u)^l(t,\cdot),\phi)\, dW^l_t,
  \end{split}
\end{equation}
 holds for all $\tau\in[0,T]$ with probability 1. As usual, we also call $(u,\bD u)$
 a solution pair of BSPDE \eqref{BSPDE}.
\end{defn}

\begin{rmk}\label{rmk modification solution}
  Assume that $(u,v)$ belongs to $ \bH_p^{n+2}\times
  \bH_p^{n+1}$ with
  $u(T,\cdot)\in L^p(\Omega,\sF_T,B_{p,p}^{n+2-2/p}),$ and further that
the equality
\begin{equation}\label{def of solution1}
  \begin{split}
  (u(\tau,\cdot),\phi)=&(G,\phi)+\int_{\tau}^{T}( a^{ij}(t,\cdot)u_{x^i x^j}(t,\cdot) +
       \sigma^{ik}(t,\cdot)v^{k}_{x^i}(t,\cdot)
          +F(u,v,t,\cdot),\phi)\, dt\\
  &-\int_{\tau}^{T}(v^l(t,\cdot),\phi)\, dW^l_t,\quad
  \forall (t, \phi) \in [0,T)\times C^{\infty}_{c}
  \end{split}
\end{equation}
 holds with probability 1.
Then by Lemma \ref{lem BSDE}, $u$ has a modification, still denoted
by itself, such that the pair $(u,v)\in \bH_{p,\infty}^{n}\times
\bH_{p,2}^{n}$ solves the Banach space-valued BSDE \eqref{bsde} with
$F(t, \cdot):= a^{ij}(t,\cdot)u_{x^i x^j}(t,\cdot)
+\sigma^{ik}(t,\cdot)v^{k}_{x^i}(t,\cdot)+F(u(t,
\cdot),v(t,\cdot),t,\cdot), t\in [0,T]$, belonging to $\bH^{n+2}_p.$
Hence, by
 Lemma \ref{lem BSDE} for any $\phi \in C^{\infty}_{c},$
  \eqref{def of solution1} holds
for all $\tau\in[0,T]$ with probability 1. Hence $u\in \sH_p^n.$
\end{rmk}

Note that Definition \ref{definition of solution} includes as a particular
case the notion of strong solution to deterministic parabolic PDEs. For
example, consider the particular case:
\begin{equation}\label{Deterministic equation}
  \left\{\begin{array}{l}
    -\frac{\partial}{\partial t} u=\Delta u+f,\\
    u(T)=u_T.
  \end{array}\right.
\end{equation}
By reversing the time, we have the following proposition
(see~\cite{Ladyzhenskaia_68}).
\begin{prop}\label{THEOREM PDE}
  For any $f\in L^p([0,T]\times \bR^d),$ and $u_T\in B_{p,p}^{2-2/p}$ with $p\in (1,\infty),$
   there
  exists a unique solution $u\in W^{1,2}_p(T)$ to Equation \eqref{Deterministic
  equation} with terminal data $u(T)=u_T.$ In addition,
  \begin{eqnarray}\label{deterministic a}
  \begin{split}
    \|u\|_{W^{1,2}_p}\leq C(d,p,T)(\|f\|_{L^p((0,T)\times \bR^d)}+
        \|u_T\|_{B_{p,p}^{2-2/p}}),
  \end{split}
\end{eqnarray}
%
  where $$\|u\|_{W^{1,2}_p}:=\|u_{xx}\|_{L^p((0,T)\times \bR^d)}
        +\|u_x\|_{L^p((0,T)\times \bR^d)}+\|u\|_{L^p((0,T)\times \bR^d)}
        +\|u_t\|_{L^p((0,T)\times \bR^d)}.$$
\end{prop}
In Proposition \ref{THEOREM PDE}, the sapce $W^{1,2}_p$ can be replaced
with $\sH_p^2$ in an equivalent way. This fact also explains why the Besov
space $B_{p,p}^n$ is used for the terminal value in Definition
\ref{stochastic space}.

\subsection{The case of space-invariant leading coefficients}

Consider the following BSPDE
\begin{equation}\label{independent of space}
  \left\{\begin{array}{l}
    \begin{split}
      -d u (t,x)=&\big[
                    a^{ij}(t)u_{x^i x^j}(t,x) +\sigma^{ik}(t)v^{k}_{x^i}(t,x)
                       +F(t,x) \big]dt\\
      &-v^l(t,x)\, dW^l_t,\quad (t,x)\in [0,T]\times \bR^{d};
    \end{split}\\
    \begin{split}
    u(T,x)=G(x),\quad x\in \bR^{d}
    \end{split}
  \end{array}\right.
\end{equation}
where $(F,G)\in \bH_{p}^{n} \times
                L^{p}(\Omega,\sF_{T},H_p^{n+1}),$ with $p\in(1,2]$ and
                $n\in\bR.$

\begin{thm}\label{thm independent space}
Assume that the coefficients $a^{ij}$ and $\sigma^{il}$
 $i,j=1,\dots,d,l=1,\dots,m,$  are $\sP$-measurable real-valued functions
which are defined on $\Omega\times [0,T]$ and bounded by a positive
constant $\Lambda,$ and also that they satisfy the
$super\textrm{-}parabolicity$ condition \ref{ass1}. Take $(F,G)\in
\bH_{p}^{n} \times
                L^{p}(\Omega,\sF_{T},H_p^{n+1}),p\in(1,2],n\in\bR.$
    Then, we have

    (i) BSPDE \eqref{independent of space} has a unique solution
    $u\in\sH_{p}^{n+2}$ and for this solution, we have
    $$\|u\|_{\sH_{p}^{n+2}}\leq C(T,n,d,p,\lambda,\Lambda)
    \left(
      \|G\|_{L^{p}(\Omega,\sF_{T},H_p^{n+1})}+
          \|F\|_{\bH_{p}^{n}}
    \right);$$

    (ii)~we have $u\in C([0,T],H_{p}^{n})$ almost surely and
    $$ \|u\|_{\bH_{p,\infty}^{n}}+\|\bD u\|_{\bH_{p,2}^{n}}\leq
            C(T,n,d,p,\lambda,\Lambda)
            \left(\|F\|_{\bH^{n}_{p}}+\|G\|_{L^{p}(\Omega,\sF_{T},H_p^{n+1})}
        \right);$$

    (iii) in particular, for the case $G\equiv 0$, there is a constant
     $C(d,p,\lambda,\Lambda)$ which does not depend on $T,$ such that
    $$ \|u_{xx}\|_{\bH_{p}^{n}}+\|(\bD u)_{x}\|_{\bH_{p}^{n}} \leq
                        C(d,p,\lambda,\Lambda)~\|F\|_{\bH_{p}^{n}},
    \|u\|_{\sH_{p}^{n+2}} \leq
                            C(d,p,\lambda,\Lambda)~\|F\|_{\bH_{p}^{n}}
                        .
    $$
\end{thm}

In view of Lemma \ref{lem BSDE} and Remark \ref{rmk modification solution},
the assertions for $p=2$ can be deduced from
 \cite{DuMeng09,Hu_Peng_91,Zhou_92}, while Theorem \ref{thm independent
space} for $p\in(1,2)$ seems to be new.
The proof of  Theorem~\ref{thm independent space} will appeal to a
harmonic analysis result which is due to Krylov \cite[Theorem
2.1]{Krylov_94}.
\begin{lem}\label{Harmonic result 1}
  Let $H$ be a Hilbert space, $p\in[2,\infty),$ $-\infty \leq a <b\leq \infty,$
  $g\in L^p((a,b)\times \bR^d,H).$ Then
\begin{eqnarray}
  \begin{split}
  \int_{\bR^d} \int_a^b \big[
    \int_a^t |\nabla T_{t-s}g(s,\cdot)(x)|_H^2ds
            \big]^{p/2}dtdx \leq
            C(d,p)\int_{\bR^d} \int_a^b |g(t,x)|_H^pdtdx
  \end{split}
\end{eqnarray}
where $T_t:=e^{\Delta t},~t\geq 0,$ is the semigroup corresponding to the
heat equation $\frac{\partial u}{\partial t}=\Delta u $ in $\bR^d.$
\end{lem}

\begin{rmk}
The assertion of Lemma~\ref{Harmonic result 1} is not true for
$p<2.$
\end{rmk}

We  have the following more general version.

\begin{prop}\label{lem 1 independent of space variable}
    Let $a^{ij}(t)$ satisfy the strong ellipticity condition, i.e.
    there exit two positive constants $\lambda_1$ and $\Lambda_1$ such that
\begin{equation}
    \Lambda_1 |\xi|^2 \geq a^{ij}(t)\xi^i \xi^j\geq \lambda_1 |\xi|^2
\end{equation}
    holds for all $\xi\in\bR^d,t\geq 0$ with probability 1.
Assume that
    $g\in \bH_{p}^{n}$ with $p\in[2,\infty)$ and $n\in\bR.$
    Then, the SPDE
    \begin{equation}\label{spde harmonic1}
        \left\{\begin{array}{l}
        \begin{split}
          d \eta (t,x)=
              a^{ij}(t)\eta_{x^i x^j}(t,x)
                        dt+
                        g^l(t,x)dW^l_t,\quad
           ~~~ (t,x)\in [0,T]\times \bR^{d},
       \end{split}\\
      \begin{split}
       \eta(0,x)=0,~~~~~~x\in \bR^{d},
             \end{split}
     \end{array}\right.
    \end{equation}
    has a unique solution $\eta\in \bH^{n+1}_p$ such that
       for any $\phi \in C^{\infty}_{c}$, the equality
\begin{equation}\label{1003_7.4}
  \begin{split}
  &(\eta(\tau,\cdot),\phi)=\int^{\tau}_{0}( a^{ij}(t)\eta_{x^i
  x^j}(t,\cdot),\phi)
          dt+\int^{\tau}_{0}(g^l(t,\cdot),\phi)dW^l_t,
  \end{split}
\end{equation}
 holds for all $\tau\in(0,T]$ with probability 1, and there holds the
 following estimate
 \begin{equation}\label{SPDE esti harmonic1}
  \begin{split}
 \| \eta _x \| _{\bH^{n}_p} \leq C(d,p,\lambda_1,\Lambda_1)
                \|g\|_{\bH_{p}^{n}}.
  \end{split}
\end{equation}
\end{prop}

\begin{proof}
  In view of  \cite[Theorem 4.10]{Krylov_99}, SPDE \eqref{spde harmonic1}
  has a unique solution.
  It remains to prove the estimate \eqref{SPDE esti harmonic1}.
    It is sufficient to prove the estimate for $n=0,$ and other cases can
    be proved by induction.

We follow a standard procedure which is due to Krylov (for instance, see
\cite[Theorem 4.10 pp. 205-206]{Krylov_99}).


  First, for the model case $a:=(a_{ij})_{1\leq i,j \leq d}=I$,
   it can be checked that
  $$\eta(t,x)=
  \int_0^t T_{t-r}g^l (r,x)\, dW^l_r \quad a.s.,
  $$
and thus,
      $$\eta_x(t,x)=
      \int_0^t \nabla T_{t-r}g^l (r,x)\, dW^l_r \quad a.s.,
  $$
where $T_t:=e^{\Delta t},~t\geq 0,$ is the semigroup corresponding to the
heat equation $\frac{\partial u}{\partial t}=\Delta u $ in $\bR^d.$ From
 Lemma \ref{Harmonic result 1}, we get
\begin{eqnarray*}
  \begin{split}
    \|\eta_x\|_{\bH_p^{0}}&=E  \int_{\bR^d} \int_0^T
    \left|\int_0^t \nabla T_{t-r}g^l (s,x)dW^l_s\right|^pdtdx    \\
    &\leq C(p)E  \int_{\bR^d} \int_0^T \big[
        \int_0^t |\nabla T_{t-s}g(s,\cdot)(x)|^2ds
            \big]^{p/2}dtdx       \\
            & \leq
            C(d,p)\|g\|_{\bH_p^0 }.
  \end{split}
\end{eqnarray*}
 For the general case, we can take $a \geq I,$ otherwise
  we take a nonrandom time change. Take $\sigma (t)=\sigma^*(t) \geq 0$
 as a solution of the matrix equation $\sigma^2(t)+2I=2a.$ Furthermore, we also
  assume that there is a d-dimensional Wiener
 process $(B_t)_{t\geq 0}$ independent of $(\sF_t)_{0\leq T}.$

Then, like the model case, the equation
$$d\zeta(t,x)=\Delta \zeta(t,x) dt +g^l(t,x-\int_0^t\sigma(s)\, dB_s)\, dW^l_t,$$
with the zero initial condition has a unique solution $\zeta\in \bH_p^0$
satisfying \eqref{1003_7.4} and \eqref{SPDE esti harmonic1}. Note that the
predictable $\sigma$-algebra $\sP$ is replaced by $\sigma$-algebra
generated by $\sF_t\vee \sigma(B_s;s\leq t)$ here.
  In
particular, as our norms are all translation invariant with respect to the
space variable, we have
$$ \|\zeta_x\|_{\bH_p^{0}}\leq C(d,p)\|g\|_{\bH_p^0 }.$$
The application of It\^o-Wentzell formula (c.f \cite{Krylov_09})
shows that the field $Y(t,x):=\zeta(t,x+\int_0^t\sigma(s)dB_s),
(t,x)\in [0,T]\times \bR^d$ solves the SPDE
$$ dY(t,x)=a^{ij}(t)Y_{x^i x^j}(t,x)\,dt+ g^l(t,x)\,dW^l_t
            +Y_{x^i}(t,x)\sigma^{ij}(t) \, dB^j_t, \quad Y(0,x)=0.$$
For any $\phi \in C^{\infty}_{c}$ and $t\geq 0,$
$$ (\eta(t,\cdot),\phi)=E[(Y(t,\cdot),\phi)|\sF_t]=
        E\big[ (\zeta(t,\cdot+\int_0^t\sigma(s)dB_s),\phi)|\sF_t  \big] \quad a.s..
    $$
Therefore,
$$(\eta_x(t,\cdot),\phi)=E\big[ (\zeta_x(t,\cdot+\int_0^t\sigma(s)dB_s),\phi)|\sF_t
                     \big]    \quad a.s..
  $$
As $ C^{\infty}_{c}$ is separable and dense in $H_{p/(p-1)}^0$, it follows
that
 $$   \|\eta_x(t,\cdot)\|_{H_p^0}^p \leq E\big[ \|\zeta_x(t,\cdot)\|_{H_p^0}^p|
   \sF_t\big] \quad a.s..
    $$
Hence, $$\|\eta_x\|_{\bH_p^0}\leq \|\zeta_x\|_{\bH_p^0}
           \leq C(d,p)\|g\|_{\bH_p^0 } .$$
By considering the possible nonrandom time change, we get \eqref{SPDE esti
harmonic1} for $n=0.$
\end{proof}

%
%

\begin{proof}[Proof of Theorem \ref{thm independent space}] Without
loss of generality, assume that $m=1$.
\emph{Step 1}. We use the duality method and Proposition~\ref{SPDE
esti harmonic1} to prove assertion (i).

Consider the following SPDE:
\begin{equation}\label{spde of the thm indep_varialble}
  \left\{\begin{array}{l}
    \begin{split}
      d \eta (t,x)=&[
              a^{ij}(t)\eta_{x^i x^j}(t,x)
                  +f(t,x)]dt\\
                 &~~~~+[-\sigma^{i}(t)\eta_{x^i}(t,x)+g(t,x)]dW_t,\quad
           ~~~ (t,x)\in [0,T]\times \bR^{d},
    \end{split}\\
    \begin{split}
       \eta(0,x)=0,~~~~~~x\in \bR^{d},
    \end{split}
  \end{array}\right.
\end{equation}
where $(f,g)\in (\bH_{2}^{-n-2}\cap \bH_{p'}^{-n-2} )\times
                 (\bH_{2}^{-n-1}\cap  \bH_{ p'}^{-n-1}),$ and $1/p+1/p'=1.$

Then it follows form \cite[Theorem 4.10]{Krylov_99} that SPDE \eqref{spde
of the thm indep_varialble} has a unique solution $u\in \bH_q^{-n}$ which
satisfies
\begin{eqnarray} \label{BSPDE 1}
  \begin{split}
    \|\eta\|_{\bH_{q}^{-n}}&\leq
        C(T,d,q,\lambda,\Lambda)\left(\|f\|_{\bH_{q}^{-n-2}}
                +\|g\|_{\bH_{q}^{-n-1}}\right),\\
    \|\eta_{xx}\|_{\bH_{q}^{-n-2}} &\leq
        C(d,q,\lambda,\Lambda) \left(
                        \|f\|_{\bH_{q}^{-n-2}}
                            +\|g\|_{\bH_{q}^{-n-1}}   \right),\\
    E\sup_{t\in [0,T]}\|\eta(t,\cdot)\|_{H_q^{-n-1}}
     &\leq C(T,d,q,\lambda,\Lambda)
        \left(
                        \|f\|_{\bH_{q}^{-n-2}}
                            +\|g\|_{\bH_{q}^{-n-1}}   \right)
  \end{split}
\end{eqnarray}
where $q=2$ or $p'.$ For the moment, assume that
$$(F,G)\in (\bH_{p}^{n}\cap \bH_{2}^{n} )
\times (L^{p}(\Omega,\sF_{T},H_{p}^{n+1})\cap
             L^{2}(\Omega,\sF_{T},H_{2}^{n+1})).$$

For $p=2$  BSPDE \eqref{independent of space} has a unique pair $(u,v)\in
\bH_{2}^{n+2} \times \bH_2^{n+1}$ such that (see \cite{Zhou_92})
    $$ \|u\|_{\bH_{2}^{n+2}}+ \|v\|_{\bH_2^{n+1}} \leq
        C(T,d,\lambda,\Lambda) [\|F\|_{\bH_2^{n}}
            +\|G\|_{L^{2}(\Omega,\sF_{T},H_{2}^{n+1})}],$$
and for any $\phi \in C^{\infty}_{c}$ and $\tau\in [0,T]$
\begin{equation}\label{1003_3_25}
  \begin{split}
  (u(\tau,\cdot),\phi)=&(G,\phi)+\int_{\tau}^{T}( a^{ij}(t)u_{x^i x^j}(t,\cdot) +
       \sigma^{i}(t)v_{x^i}(t,\cdot)
          +F(t,\cdot),\phi)\, dt\\
  &\int_{\tau}^{T}(v(t,\cdot),\phi)\, dW_t, \quad
  a.s..
  \end{split}
\end{equation}
 Furthermore, keeping in mind the existence of $(u,v)\in \bH_{2}^{n+2}
\times \bH_2^{n+1},$
 we conclude that (at
least for a modification of $u$)  for any $\phi \in C^{\infty}_{c}$, the
equality \eqref{1003_3_25}
 holds for all $\tau\in[0,T]$ with probability 1.
 From Remark \ref{rmk modification solution}, we have $u\in\sH_2^{n+2}.$

 The parallelogram rule yields the following
\begin{eqnarray*}
  \begin{split}
    &\int_{\bR^{d}}(1-\Delta)^{\frac{n+1}{2}}u(t,x)
            (1-\Delta)^{-\frac{n+1}{2}}\eta(t,x)dx                           \\
    =\, &\frac{1}{4}\{\|(1-\Delta)^{\frac{n+1}{2}}u(t,\cdot)
            +(1-\Delta)^{-\frac{n+1}{2}}\eta (t,\cdot)\|_{L^2(\bR^d)}^2
            +\\
    &\|(1-\Delta)^{\frac{n+1}{2}}u(t,\cdot)
            -(1-\Delta)^{-\frac{n+1}{2}}\eta (t,\cdot)\|_{L^2(\bR^d)}^2
                   \}.
  \end{split}
\end{eqnarray*}
Applying  It\^o's formula to compute the square of the norm (see
\cite[Theorem 3.1]{Krylov_Rozovskii81}), we get
\begin{eqnarray*}
    \begin{split}
    &  E \int_{0}^{T} (u(t,\cdot),f(t,\cdot))+(v(t,\cdot),g(t,\cdot))dt     \\
       =\, &(G,\eta(T,\cdot))+E\int_{0}^{T}(F(t,\cdot),\eta(t,\cdot))dt        \\
      \leq \, &\|G\|_{L^{p}(\Omega,\sF_{T},H_{p}^{n+1})}
                    \|\eta(T)\|_{L^{p'}(\Omega,\sF_{T},H_{p'}^{-n-1})}
                +\|F\|_{\bH^{n}_{p}}    \|\eta\|_{\bH_{p'}^{-n}}            \\
       \leq \, &(\|G\|_{L^{p}(\Omega,\sF_{T},H_{p}^{n+1})}
              +\|F\|_{\bH^{n}_{p}})
            (\|\eta(T)\|_{L^{p'}(\Omega,\sF_{T},H_{p'}^{-n-1})}
                +\|\eta\|_{\bH_{p'}^{-n}})                                  \\
      \leq \, &C(T,\lambda,\Lambda,d,p)
                    (\|G\|_{L^{p}(\Omega,\sF_{T},H_{p}^{n+1})}
                        +\|F\|_{\bH^{n}_{p}})
                    (\|f\|_{\bH_{p'}^{-n-2}}+\|g\|_{\bH_{p'}^{-n-1}}).      \\
    \end{split}
\end{eqnarray*}
Note that $(F,G)\in (\bH_{p}^{n}\cap \bH_{2}^{n} ) \times
(L^{p}(\Omega,\sF_{T},H_{p}^{n+1})\cap
             L^{2}(\Omega,\sF_{T},H_{2}^{n+1})).$

For  $(F,G)\in \bH_{p}^{n} \times L^{p}(\Omega,\sF_{T},H_{p}^{n+1}),$ we
choose a sequence $(F^k,G^k)\in (\bH_{p}^{n}\cap \bH_{2}^{n} ) \times
(L^{p}(\Omega,\sF_{T},H_{p}^{n+1})\cap
             L^{2}(\Omega,\sF_{T},H_{2}^{n+1})),~k=1,2,\dots,., $ such that
\begin{eqnarray}\label{proof of thm indep space sequence}
  \begin{split}
\|F^k-F\|_{\bH_{p}^{n}}+\|G^k-G\|_{L^{p}(\Omega,\sF_{T},H_{p}^{n+1})}
\rightarrow 0 ~as ~ k \rightarrow \infty.
  \end{split}
\end{eqnarray}
 Denote by $(u_k,v_k)$ the unique
solution pair to BSPDE \eqref{independent of space} for $(F,G):=(F^k,G^k)$.
Thus,
\begin{eqnarray*}
    \begin{split}
    &  E \int_{0}^{T} (u^k(t,\cdot),f(t,\cdot))+(v^k(t,\cdot),g(t,\cdot))dt     \\
    &   =(G^k,\eta(T,\cdot))+E\int_{0}^{T}(F^k(t,\cdot),\eta(t,\cdot))dt        \\
    &   \leq C(T,d,p,\lambda,\Lambda)
                    (\|G^k\|_{L^{p}(\Omega,\sF^k_{T},H_{p}^{n+1})}
                        +\|F^k\|_{\bH^{n}_{p}})
                    (\|f\|_{\bH_{p'}^{-n-2}}+\|g\|_{\bH_{p'}^{-n-1}}).      \\
    \end{split}
\end{eqnarray*}
where $C(T,d,p,\lambda,\Lambda)$ is independent of k. Noting that
$\bH_{2}^{-n-2}\cap \bH_{p'}^{-n-2}$ and $\bH_{2}^{-n-1}\cap
\bH_{p'}^{-n-1}$ are dense in $\bH_{p'}^{-n-2}$ and $\bH_{p'}^{-n-1}$
respectively, and that $(f,g)\in(\bH_{2}^{-n-2}\cap \bH_{p'}^{-n-2})
\times(\bH_{2}^{-n-1}\cap  \bH_{p'}^{-n-1})$ is arbitrary, from the last
inequality, we have
\begin{eqnarray}\label{1003301}
  \begin{split}
\|u^k\|_{\bH_p^{n+2}}+\|v^k\|_{\bH_p^{n+1}} \leq
    C(T,d,p,\lambda,\Lambda)\left [\|F^k\|_{\bH_p^{n}}
    +\|G^k\|_{L^{p}(\Omega,\sF^k_{T},H_{p}^{n+1})}\right ].
  \end{split}
\end{eqnarray}
Moreover,
\begin{equation*}
    \begin{split}
      & \|u^k\|_{\sH^{n+2}_{p}}                                                     \\
            & =\|u^{k}_{xx}\|_{\bH_{p}^{n}}
            +\|a^{ij}u^{k}_{x^i x^j} +\sigma^{i}v^{k}_{x^i}
                    +F^{k}\|_{\bH_{p}^{n}}+\|v^{k}_x\|_{\bH_{p}^{n}}
                            +\|G^{k}\|_{L^{p}(\Omega,\sF_{T},B_{p,p}^{n+2-2/p})}        \\
      &\leq C(n,d,p,\lambda,\Lambda)
                [   \|u^{k}_{xx}\|_{\bH_{p}^{n}}+\|v^{k}\|_{\bH_{p}^{n+1}}
                    +\|F^{k}\|_{\bH_{p}^{n}}
                    +\|G^{k}\|_{L^{p}(\Omega,\sF_{T},H_{p}^{n+1})}  ]              \\
      &\leq C(n,d,p,\lambda,\Lambda)
                [  \|u^{k}\|_{\bH_{p}^{n+2}}   +\|v^{k}\|_{\bH_{p}^{n+1}}
                    +\|F^{k}\|_{\bH_{p}^{n}}
                    +\|G^{k}\|_{L^{p}(\Omega,\sF_{T},H_{p}^{n+1})}   ]             \\
      &\leq C(T,n,d,p,\lambda,\Lambda)
                [\|F^{k}\|_{\bH_{p}^{n}}
                    +\|G^{k}\|_{L^{p}(\Omega,\sF_{T},H_{p}^{n+1})}   ]             \\
    \end{split}
\end{equation*}
and this combined with  $u^k\in \sH_2^{n+2},$ implies $u^k\in \sH_p^{n+2}$
for $k=1,2,3,\dots.$

From \eqref{proof of thm indep space sequence}, \eqref{1003301} and
the last inequality, it follows that $u^k$ is a Cauchy sequence in
$\sH_p^{n+2}.$ By Theorem \ref{stochastic banach property thm},
there exists $u \in \sH_p^{n+2}$
 such that $\|u^k-u\|_{\sH_p^{n+2}}\rightarrow 0,~as~ k \rightarrow \infty,$
 and there holds the following estimate
$$\|u\|_{\sH_{p}^{n+2}}\leq C(T,n,p,d,\lambda,\Lambda)
    [
      \|G\|_{L^{p}(\Omega,\sF_{T},H_{p}^{n+1})}+
          \|F\|_{\bH_{p}^{n}}
    ].$$
 Denote $v:=\bD u.$ It is obvious that
$\|v^k-v\|_{\bH_p^{n+1}}\rightarrow 0,~as~ k \rightarrow \infty.$ In view
of Remark \ref{rmk 2}, one can check that $v\in\bH_p^{n+1}\cap
\bH_{p,2}^{n}.$ By taking limits one can check that $u \in \sH_p^{n+2}$ is
a solution of BSPDE \eqref{independent of space}.

Now we prove the uniqueness of the solution. Suppose that $F=0,$
$G=0$ and $u\in\sH_p^{n+2}$ solving BSPDE \eqref{independent of
space}. It is sufficient to show $u=0$, which is immediate from the
last estimate with $F=0$ and $G=0$.



\emph{Step 2}. We prove assertion (ii).

Note that $L^{p}(\Omega,\sF_{T},H_{p}^{n+1}) $ is continuously embedded
into $ L^{p}(\Omega,\sF_{T},H_{p}^{n}).$ From Lemma \ref{lem BSDE}, it
follows that $u\in \bH_{p,\infty}^{n} ,v\in \bH_{p,2}^{n},$ and $u\in
C([0,T],H_p^n)$ almost surely. In fact, in view of Lemma \ref{lem BSDE} and
 Theorem \ref{stochastic banach property thm}, we have
\begin{equation*}
    \begin{split}
      & \|u\|_{\bH^{n}_{p,\infty}}+\|v\|_{\bH_{p,2}^{n}}                    \\
             \leq\, & C(T,p)\left(
            \|a^{ij}u_{x^i x^j} +\sigma^{ik}v^{k}_{x^i}
                    +F\|_{\bH_{p}^{n}}
                            +\|G\|_{L^{p}(\Omega,\sF_{T},H_{p}^{n})}\right)       \\
      \leq \, & C(T,p,\lambda,\Lambda)
                \left(   \|u_{xx}\|_{\bH_{p}^{n}}+\|v_x\|_{\bH_{p}^{n}}
                    +\|F\|_{\bH_{p}^{n}}
                    +\|G\|_{L^{p}(\Omega,\sF_{T},H_{p}^{n})}  \right)              \\
      \leq \, &C(T,p,\lambda,\Lambda)
                \left(  \|u\|_{\sH_{p}^{n+2}}
                    +\|F\|_{\bH_{p}^{n}}
                    +\|G\|_{L^{p}(\Omega,\sF_{T},H_{p}^{n})}   \right)             \\
      \leq \, & C(T,n,d,p,\lambda,\Lambda)
                \left(\|F\|_{\bH_{p}^{n}}
                    +\|G\|_{L^{p}(\Omega,\sF_{T},H_{p}^{n+1})}  \right).             \\
    \end{split}
\end{equation*}

\emph{Step 3}. We prove assertion (iii) using the duality method.

Consider $G=0$. For $(f,g)\in (\bH_{p'}^{-n}\cap\bH_{2}^{-n})
                \times (\bH_{p'}^{-n+1}\cap\bH_{2}^{-n+1})$, the
                Hessian $\eta_{xx}$ of the corresponding solution solves SPDE~\eqref{spde of the thm
indep_varialble} with $(f,g)$ being replaced with $(f_{xx},g_{xx})$.
For the SPDE with $(f,g)$, we have the following analogue to
~\eqref{BSPDE 1}:
\begin{equation*}
    \begin{split}
        \|\eta_{xx}\|_{\bH_{p'}^{-n}}\
        &\leq\  C(d,p,\lambda,\Lambda) \left(
                   \|f\|_{\bH_{p'}^{-n}}
                            +\|g\|_{\bH_{p'}^{-n+1}}   \right).
    \end{split}
\end{equation*}
Furthermore, proceeding identically as in the proof of assertion
(i), we have
\begin{eqnarray*}
    \begin{split}
    &E \int_{0}^{T}
    (u_{x^ix^j}(t,\cdot),f(t,\cdot))+(v_{x^ix^j}(t,\cdot),g(t,\cdot))dt\\
    =&E \int_{0}^{T} (u(t,\cdot),f_{x^ix^j}(t,\cdot))+(v(t,\cdot),g_{x^ix^j}(t,\cdot))dt   \\
       =&E\int_{0}^{T}(F(t,\cdot),\eta_{x^ix^j}(t,\cdot))dt                            \\
       \leq& \|F\|_{\bH^{n}_{p}}    \|\eta_{xx}\|_{\bH_{p'}^{-n}}                  \\
       \leq& C(\lambda,\Lambda,d,p)
                 \|F\|_{\bH^{n}_{p}} (
                    \|f\|_{\bH_{p'}^{-n}}+\|g\|_{\bH_{p'}^{-n+1}}),~\textrm{ for }i,j=1,\dots,d.                \\
    \end{split}
\end{eqnarray*}
Hence, by the arbitrariness of $(f,g)$ and the denseness of
$(\bH_{p'}^{-n}\cap\bH_{2}^{-n})
                \times (\bH_{p'}^{-n+1}\cap\bH_{2}^{-n+1})$ in $\bH_{p'}^{-n} \times
                \bH_{p'}^{-n+1}$
it follows that
\begin{eqnarray}\label{1}
  \begin{split}
    \|u_{xx}\|_{\bH_{p}^{n}}+\|v_{xx}\|_{\bH_{p}^{n-1}} \leq
                        C(d,p,\lambda,\Lambda)~\|F\|_{\bH_{p}^{n}}.
  \end{split}
\end{eqnarray}

On the other hand, let $\zeta(t,x):=u(t,x+\int_0^t \sigma(s) dW_s)$.
Applying the It\^o-Wentzell formula (cf. \cite{Krylov_09}), we have
\begin{equation}
  \left\{\begin{array}{l}
    \begin{split}
      -d \zeta (t,x)=&\big[
                    (a^{ij}(t)-\alpha^{ij}(t))\zeta(t,x)_{x^i x^j}
                       +F(t,x+\int_0^t \sigma(s) dW_s) \big]dt\\
      &-[\sigma^i \zeta_{x^i}(t,x)+ v(t,x+\int_0^t \sigma(s) dW_s)]
      dW_t,\quad (t,x)\in [0,T]\times \bR^{d};
    \end{split}\\
    \begin{split}
    \zeta(T,x)=0, \quad x\in \bR^{d}.
    \end{split}
  \end{array}\right.
\end{equation}
We consider the dual SPDE
\begin{equation}\label{1003302}
  \left\{\begin{array}{l}
    \begin{split}
      d \psi (t,x)=&
              (a^{ij}(t)-\alpha^{ij}(t))\psi_{x^i x^j}(t,x)\,
                  dt+h(t,x)\,dW_t,\quad (t,x)\in [0,T]\times
                  \bR^{d};
    \end{split}\\
    \begin{split}
       \psi(0,x)=0,\quad x\in \bR^{d}
    \end{split}
  \end{array}\right.
\end{equation}
where $h\in\bH^{-n}_{p'} \cap \bH^{-n}_{2},$ $1/p'+1/p=1.$ In view of
Proposition \ref{lem 1 independent of space variable}, we conclude that
SPDE \eqref{1003302} has a unique solution $\psi \in \bH_{p'}^{-n+1}$
satisfying
$$\|\psi_x\|_{\bH_{p'}^{-n}} \leq C(d,p,\lambda,\Lambda)
\|h\|_{\bH^{-n}_{p'}}.$$ Moreover, we have
\begin{eqnarray}
  \begin{split}
    &E\int_0^T(\sigma^i \zeta_{x^ix^j}(t,\cdot)+ v_{x^j}(t,\cdot+\int_0^t \sigma(s)\, dW_s),
                h(t,\cdot)) \,dt                                                  \\
    &=E\int_0^T(\sigma^i \zeta_{x^i}(t,\cdot)+ v(t,\cdot+\int_0^t \sigma(s)\, dW_s),
                h_{x^j}(t,\cdot))\, dt                                           \\
    &=E\int_0^T(\psi_{x^j}(t,\cdot),F(t,\cdot+\int_0^t \sigma(s)\, dW_s))\, dt     \\
    &\leq \|\psi_{x}\|_{\bH_{p'}^{-n}}
        \|F(\cdot,\cdot+\int_0^t \sigma(s) \, dW_s(\cdot))\|_{\bH_p^n}           \\
    &\leq C(d,p,\lambda,\Lambda)
            \|h\|_{\bH_{p'}^{-n}} \|F\|_{\bH_p^n} \quad for~ j=1,\dots,d.
  \end{split}
\end{eqnarray}
Since $h$ is arbitrary and $\bH^{-n}_{p'} \cap \bH^{-n}_{2}$ is dense in
 $\bH^{-n}_{p'},$
we have
$$\|\sigma^i\zeta_{x^ix}(\cdot,\cdot)+ v_{x}(\cdot,\cdot+\int_0^t
\sigma(s)\,  dW_s)  \|_{\bH_p^n} \leq C(d,p,\lambda,\Lambda)
\|F\|_{\bH_p^n},$$
 which yields
$$\|\sigma^i u_{x^ix}+ v_{x} \|_{\bH_p^n}
 \leq C(d,p,\lambda,\Lambda) \|F\|_{\bH_p^n}.$$
Therefore,
\begin{eqnarray}\label{2}
  \begin{split}
    \|v_x\|_{\bH_p^n}
    &\leq \|\sigma^i u_{x^ix}\|_{\bH_p^n}+\|\sigma^i u_{x^ix}+ v_{x}
                    \|_{\bH_p^n}
    \leq C(d,p,\lambda,\Lambda)\|F\|_{\bH_p^n},
  \end{split}
\end{eqnarray}
which, combined with \eqref{1}, implies the assertion (iii).

%
%

  The proof is complete.
\end{proof}

\begin{rmk}\label{rmk p q}
    If the assumptions of Theorem \ref{thm independent space} are
     satisfied for  both $q_1$ and $q_2$ instead of $p,$ where $q_1,q_2\in (1,2],$
    then the solutions in $\sH_{q_1}^{n+2}$ and $\sH_{q_2}^{n+2}$
    coincide. Indeed, we need only to take $(F^k,G^k)\in (\bH_{q_1}^{n}\cap \bH_{2}^{n}
    \cap \bH_{q_2}^n ) \times
(L^{q_1}(\Omega,\sF_{T},H_{q_1}^{n+1})\cap
             L^{2}(\Omega,\sF_{T},H_{2}^{n+1})
             \cap L^{q_2}(\Omega,\sF_{T},H_{q_2}^{n+1})) $
   during the proof of Theorem \ref{thm independent
   space}. Then the approximating solutions in $\sH_{q_1}^{n+2}$
 and $\sH_{q_2}^{n+2}$ coincide in $\sH_{2}^{n+2}.$ This implies
the solutions to  \eqref{independent of space} in $\sH_{q_1}^{n+2}$
 and $\sH_{q_2}^{n+2}$ coincide.
\end{rmk}

\begin{rmk}\label{p>2 simple case}
For the case $p\in(2,\infty),$ consider the following BSPDE
\begin{equation}\label{p>2 simple bspde}
  \left\{\begin{array}{l}
    \begin{split}
      -d u (t,x)=&\big[
                    a^{ij}(t)u_{x^i x^j}(t,x) +\sigma^{ik}(t)v_{x^i}(t,x)+F(t,x)
                    \big]dt\\
                    &-v^k(t,x)dW^k_t,\quad~~(t,x)\in [0,T]\times \bR^{d},
    \end{split}\\
    \begin{split}
    u(T,x)=G(x),\quad~~~~x\in \bR^{d},
    \end{split}
  \end{array}\right.
\end{equation}
and  SPDE:
\begin{equation}\label{p>2 simple spde}
  \left\{\begin{array}{l}
    \begin{split}
      d \eta (t,x)=&[
              a^{ij}(t)\eta_{x^i x^j}(t,x)
                  +f(t,x)]dt      \\
                  &-\sigma^{ik}(t)\eta_{x^i}(t,x)dW^k_t
                    \quad~ (t,x)\in [0,T]\times \bR^{d},
    \end{split}\\
    \begin{split}
       \eta(0,x)=0,~~~~~~x\in \bR^{d},
    \end{split}
  \end{array}\right.
\end{equation}
where $f\in \bH_{p'}^{-n-2} ,$ $(F,G)\in \bH_{p}^{n} \times
                L^{p}(\Omega,\sF_{T},H_p^{n+1}),$ $1/p+1/p'=1,$
                $n\in\bR$ and $(a^{ij})_{1\leq i,j\leq d}$ and $(\sigma^{ik})_{1\leq i\leq d,
                1\leq k \leq m}$ are the same as
                Theorem \ref{thm independent space}.
If $a= I$ and $\sigma=0,$  one can check that
$\eta(t,x)=\int_0^te^{\Delta (t-s)}f(s,x)ds \in\bH_{p'}^{-n}$ is the
unique solution of \eqref{p>2 simple spde} in the sense of
\cite{Kryl96,Krylov_99}. For the general $a$ and $\sigma,$ by
applying the the It\^o-Wentzell formula and the technical method
used in Proposition \ref{lem 1 independent of space variable}, we
can conclude that \eqref{p>2 simple spde} has a unique solution
$\eta\in\bH_{p'}^{-n}$. It is crucial that $\sigma$ is invariant in
the space variable.

  Then through a
 procedure similar to the proof of Theorem \ref{thm independent space},
  we can conclude that BSPDE \eqref{p>2 simple bspde} has a unique solution pair
$(u,v)$ such that $u\in\bH_p^{n+2}\cap\bH_{p,\infty}^n,$ $
v(\cdot,\cdot+\int_0^{\cdot}\sigma^k(s)dW_s^k)\in\bH_{p,2}^n$
 and  for any $\phi \in C^{\infty}_{c},$ the equality
\begin{equation*}
  \begin{split}
  &(u(\tau,\cdot),\phi)=(G,\phi)+\int_{\tau}^{T}( a^{ij}(t)u_{x^i x^j}(t,\cdot) +
       \sigma^{ik}(t)v^{k}_{x^i}(t,\cdot)
          +F(t,\cdot),\phi)\,dt
  -\int_{\tau}^{T}(v^k(t,\cdot),\phi)\,dW^k_t,
  \end{split}
\end{equation*}
 holds for all $\tau\in[0,T]$ with probability 1.
For this solution pair, we have $u\in C([0,T],H_{p}^{n})$ almost surely and
\begin{eqnarray*}
  \begin{split}
    \|u\|_{\bH_{p}^{n+2}}+\|u\|_{\bH_{p,\infty}^{n}}+\|v'\|_{\bH_{p,2}^{n}}
        \leq C(T,n,d,p,\lambda,\Lambda)
    \left(
      \|G\|_{L^{p}(\Omega,\sF_{T},H_p^{n+1})}+
          \|F\|_{\bH_{p}^{n}}
    \right)
  \end{split}
\end{eqnarray*}
where  $v'=v(\cdot,\cdot+\int_0^{\cdot}\sigma^k(s)dW_s^k).$  In particular,
when $G=0,$ we have $    \|u\|_{\bH_{p}^{n+2}} \leq
                            C(d,p,\lambda,\Lambda)~\|F\|_{\bH_{p}^{n}}.$
\end{rmk}

\subsection{The case of general variable leading coefficients}
Now we deal with the general case.
\begin{thm}\label{thm general case}
Suppose that the assumptions \ref{ass1}-\ref{ass4} are all satisfied.
Consider $ G \in L^{p}(\Omega,\sF_{T},H_{p}^{n+1})$ with $p\in(1,2]$
 and $n\in\bR.$ Then BSPDE \eqref{BSPDE} has a unique
solution $u\in \sH_p^{n+2},$ satisfying the following inequality
\begin{eqnarray}
  \begin{split}
    \|u\|_{\sH_p^{n+2}}\leq C(T,n,\kappa,\varrho,d,p,\lambda,\Lambda)
        \left ( \|F(0,0,\cdot,\cdot)\|_{\bH_p^n }+\|G\|_{L^{p}(\Omega,\sF_{T},H_{p}^{n+1})}
        \right ).
  \end{split}
\end{eqnarray}
\end{thm}

The following lemma can be found in \cite[Lemma 5.2]{Krylov_99}.
\begin{lem}\label{lem multiplier}
Let $\zeta \in C^{\infty}_{c}(\bR^d)$ be a nonnegative function such that
$\int \zeta(x)dx=1$ and define $\zeta_k(x)=k^d\zeta(kx),k=1,2,3,\dots.$
Then for any $u\in H_p^n,$ $p\in(1,\infty),$ and any $n\in\bR,$  we have

(i) $\|au\|_{n,p}\leq C \|a\|_{B^{|n|+\gamma}} \|u\|_{n,p}$ where
$C=C(d,p,n,\gamma);$

(ii) $\|u\ast\zeta_k\|_{n,p}\leq \|u\|_{n,p},$ $\|u- u\ast \zeta_k
\|_{n,p}\rightarrow 0$ as $k\rightarrow \infty.$
\end{lem}

Applying Lemma \ref{BSDE lem2}, we get a priori result about the
solution of BSPDE \eqref{BSPDE}, which is given in the following
lemma. It will play a key role in the proof of Theorem \ref{thm
general case} and distinguish our proof of BSPDEs  from that of
SPDEs in Krylov~\cite{Kryl96,Krylov_99}.
\begin{lem}\label{lem BSDE diffusion term estimate}
    Let $u\in\sH_{p,0}^{n+2}$ be a solution to BSPDE \eqref{BSPDE}. Let the
    assumptions \ref{ass1}-\ref{ass4} be satisfied. Then for any $\eps>0,$ there
    exists a constant $C=C(T,p,\eps)$ such that
\begin{eqnarray*}
  \begin{split}
    \|\bD u\|_{\bH_{p,2}^n(t)}   \leq
        &\eps
        [\|u_{xx}\|_{\bH_p^n(t)}+\|(\bD u)_x\|_{\bH_p^n(t)}
                +\|F(0,0,\cdot,\cdot)\|_{\bH_p^n(t)}]        \\
        &+C(T,p,\eps,\varrho,\Lambda)\|u\|_{\bH_p^n(t)},~t\in[0,T).
  \end{split}
\end{eqnarray*}
\end{lem}
\begin{proof}
 Denote $v:= \bD u.$
  By Lemma \ref{BSDE lem2},
  for any $\bar{\eps}>0,$ there exists a constant $C=C(T,p,\bar{\eps})$
such that
\begin{eqnarray*}
  \begin{split}
    \|v\|_{\bH_{p,2}^n}
     \leq & \bar{\eps}
        \|\cL u +\cM^k v^k +F(u,v,\cdot,\cdot)\|_{\bH_p^n}
            +C(T,p,\bar{\eps})\|u\|_{\bH_p^n}      \\
     \leq &\bar{\eps} C(\Lambda)(\|u_{xx}\|_{\bH_p^n}+\|v_x\|_{\bH_p^n}
                            +\|F(0,0,\cdot,\cdot)\|_{\bH_p^n})    \\
    &    +\bar{\eps} (\varrho (1)+1)(\|u\|_{\bH_p^n}+\|v\|_{\bH_p^n})
            +C(T,p,\bar{\eps})\|u\|_{\bH_p^n}.
  \end{split}
\end{eqnarray*}
Since
$$\|v\|_{\bH_p^n} \leq T^{(2-p)/2}\|v\|_{\bH_{p,2}^n}, $$
we choose $\bar{\eps}$ sufficiently small so that
$1-\bar{\eps}(\varrho(1)+1) T^{(2-p)/2}>1/2.$ Therefore,
\begin{eqnarray*}
  \begin{split}
    \|v\|_{\bH_{p,2}^n}
     \leq & 2\bar{\eps}C(\Lambda)
    [\|u_{xx}\|_{\bH_p^n}+\|v_x\|_{\bH_p^n}
                            +\|F(0,0,\cdot,\cdot)\|_{\bH_p^n}] \\
        &+2 \bar{\eps} (\varrho (1)+1)\|u\|_{\bH_p^n}
        +2 C(T,p,\bar{\eps})\|u\|_{\bH_p^n}            \\
    \leq & 2\bar{\eps}C(\Lambda)
    [\|u_{xx}\|_{\bH_p^n}+\|v_x\|_{\bH_p^n}
                            +\|F(0,0,\cdot,\cdot)\|_{\bH_p^n}]
        +C(T,p,\bar{\eps},\varrho(1))\|u\|_{\bH_p^n}.
  \end{split}
\end{eqnarray*}
This shows that the lemma is true for $t=0.$ Replacing $\bH_p^n$ with
$\bH_p^n(t),$ we can prove the lemma for any $t\in [0,T)$ similarly.
\end{proof}

 We have the following result about the perturbed leading coefficients.
\begin{thm}\label{thm perturbation}
 Let Assumptions \ref{ass1}-\ref{ass4} be satisfied. Then there exists a constant
  $\eps \in (0,1)$ depending only on $d,p,\lambda ~\textrm{and}~ \Lambda$
  such that if the inequality
\begin{eqnarray}\label{1003_3}
  \begin{split}
    &\|(a(t,\cdot)-\bar{a}(t))^{ij}(u_1)_{ij}\|_{n,p}+
       \|(\sigma(t,\cdot)-\bar{\sigma}(t,\cdot))^{ik}(v_1)^k_{i}\|_{n,p} \\
        \leq
    &\eps(  \|(u_1)_{xx}\|_{n,p}+\|(v_1)_{x}\|_{n,p}) \\
    &~ +
        K_0(\|u_1\|_{n,p}+\|v_1\|_{n,p}),~ \forall
        (u_{1},v_{1})\in H_{p}^{n+2}\times H_{p}^{n+1},t\geq 0,
  \end{split}
\end{eqnarray}
holds for some constant $K_0$ and some pair $(\bar{a},~\bar{\sigma})$ which
satisfies the assumptions in Theorem \ref{thm independent space}, there
exists a unique solution $u\in\sH_{p,0}^{n+2}$ to Equation \eqref{BSPDE}
with $G=0.$ Moreover, we have
\begin{eqnarray}\label{perturbation estimate}
  \begin{split}
    \|u\|_{\sH_{p}^{n+2}}
        \leq C(T,K_0,\varrho,d,p,\lambda,\Lambda)
            \|F(0,0,\cdot,\cdot)\|_{\bH_p^n}.
  \end{split}
\end{eqnarray}
In particular, $C$ is independent of $T$  if $K_0=0$ and
 $\varrho\equiv 0.$
\end{thm}
\begin{proof}
  \emph{Step 1.} We first prove that there is a generic constant
  $\eps\in(0,1)$ such that the inequality \eqref{1003_3} yields the
   estimate \eqref{perturbation estimate} for any solution $u\in \sH_{p,0}^{n+2}$ to
   BSPDE \eqref{BSPDE}.
  Denote $v:=\bD u$ and rewrite BSPDE \eqref{BSPDE} into the following form:
\begin{equation}
  \left\{\begin{array}{l}
    \begin{split}
      -d u (t,x)=&\big[
      \bar{\cL}u(t,x) +\bar{\cM}^kv^k(t,x)+(\cL -\bar{\cL})u(t,x)
      +(\cM-\bar{\cM})^k v^k(t,x)           \\
      &+F(u,v,t,x)\big]dt
      -v^k(t,x)dW^k_t,\quad~~~ (t,x)\in [0,T]\times \bR^{d};
    \end{split}\\
    \begin{split}
    u(T,x)=0,~~~~~~x\in \bR^{d}
    \end{split}
  \end{array}\right.
\end{equation}
where
\begin{eqnarray*}
  \begin{split}
    \bar{\cL}=\bar{a}^{ij}\frac{\partial^2}{\partial x^i\partial x^j},~
        \bar{\cM}^k=\bar{\sigma}^{ik}\frac{\partial}{\partial x^i},
         ~~k=1,\cdots,m.
  \end{split}
\end{eqnarray*}
In view of Theorem \ref{thm independent space}, we have
\begin{eqnarray*}
  \begin{split}
    &\|u\|_{\sH_{p}^{n+2}}   \\
        \leq & C(d,p,\lambda,\Lambda)
            \|(\cL -\bar{\cL})u+(\cM-\bar{\cM})^k v^k
                +F(u,v,\cdot,\cdot)\|_{\bH_{p}^{n}}          \\
    \leq& C(d,p,\lambda,\Lambda)[
        \eps (\|u_{xx}\|_{\bH_p^n}+\|v_x\|_{\bH_p^n})
            +K_0(\|u\|_{\bH_p^n}+\|v\|_{\bH_p^n})
                 +\|F(0,0,\cdot,\cdot)\|_{\bH_p^n}            \\
    &+\eps_1(\|u_{xx}\|_{\bH_p^n}+\|v_x\|_{\bH_p^n})
                +(\varrho(\eps_1)+\eps_1)(\|u\|_{\bH_p^n}+\|v\|_{\bH_p^n})
                 ]     \\
    \leq & C(d,p,\lambda,\Lambda)[(\eps+\eps_1)
        (\|u_{xx}\|_{\bH_p^n}+\|v_x\|_{\bH_p^n})
        +(K_0+\varrho(\eps_1)+\eps_1)(\|u\|_{\bH_p^n}+\|v\|_{\bH_p^n})   \\
    &    +\|F(0,0,\cdot,\cdot)\|_{\bH_p^n}].
  \end{split}
\end{eqnarray*}
Note that the above still holds if $\bH_p^n$ is replaced by $\bH_p^n(t)$
for $t\in [0,T).$ Furthermore, if  $K_0=0$ and $\varrho\equiv 0,$ the map
$F$ does not depend on $(u,v)$ and we get instead that
$$\|u\|_{\sH_{p}^{n+2}}
    \leq  C(d,p,\lambda,\Lambda)[
        \eps (\|u_{xx}\|_{\bH_p^n}+\|v_x\|_{\bH_p^n})
                 +\|F\|_{\bH_p^n}],
    $$
which implies the last assertion of Theorem \ref{thm perturbation} by
taking $\eps$ small enough such that $C(d,p,\lambda,\Lambda)\eps<1/2.$

Now, fix $t\in[0,T).$
 Then, noting that $\|v\|_{\bH_p^n(t)} \leq T^{(2-p)/2}\|v\|_{\bH_{p,2}^n}(t),$
from Lemma \ref{lem BSDE diffusion term estimate}, we conclude that for any
$\eps_2>0,$
 there exists a constant $C=C(T,p,\eps_2,\varrho)$ such that
    $$\|v\|_{\bH_{p}^n(t)}   \leq
        \eps_2
        (\|u_{xx}\|_{\bH_p^n(t)}+\|v_x\|_{\bH_p^n(t)}+\|F(0,0,\cdot,\cdot)\|_{\bH_p^n(t)})
            +C_2(T,p,\eps_2,\varrho,\Lambda)\|u\|_{\bH_p^n(t)}.
        $$
Thus, it follows that
\begin{eqnarray*}
  \begin{split}
    &\|u\|_{\sH_{p}^{n+2}(t)}   \\
    \leq &
    C_1(d,p,\lambda,\Lambda)\{[\eps+\eps_1+(K_0+\varrho(\eps_1)+\eps_1)\eps_2]
                (\|u_{xx}\|_{\bH_p^n(t)}+\|v_x\|_{\bH_p^n(t)})     \\
    &     +[(K_0+\varrho(\eps_1)+\eps_1)\eps_2+1]\|F(0,0,\cdot,\cdot)\|_{\bH_p^n(t)}
          +(K_0+\varrho(\eps_1)+\eps_1)(1+C_2(T,p,\eps_2,\varrho))\|u\|_{\bH_p^n(t)}
        \}.
  \end{split}
\end{eqnarray*}
 Taking $\eps_1=\eps,$
$\eps_2=\eps/(K_0+\varrho(\eps)+\eps+1)$ and $\eps=1/(4C_1+1),$
 we get
\begin{eqnarray*}
  \begin{split}
    \|u\|_{\sH_{p}^{n+2}(t)}
    \leq
         5 \|F(0,0,\cdot,\cdot)\|_{\bH_p^n(t)}
          +C(T,p,d,\lambda,\Lambda,K_0,\varrho(\eps))\|u\|_{\bH_p^n(t)},
  \end{split}
\end{eqnarray*}
which immediately implies the following inequality
$$\|u\|_{\sH_{p}^{n+2}(t)}^p \leq
C(T,p,d,\lambda,\Lambda,K_0,\varrho(\eps))
(\|F(0,0,\cdot,\cdot)\|_{\bH_p^n(t)}^p+\|u\|_{\bH_p^n(t)}^p).$$ Since (see
Remark \ref{rmk sotchastic banach thm})
$$E\sup_{s\in[t, T]}\|u(s,\cdot)\|^p_{H_{p}^{n}} \leq
                    C(p,T)\|u\|^p_{\sH_p^{n+2} (t)},$$
we have
\begin{eqnarray*}
  \begin{split}
         &\|u\|_{\sH_{p}^{n+2}(t)}^p                    \\
    \leq &C(T,p,d,\lambda,K_0,\Lambda,\varrho(\eps))
        \left(\|F(0,0,\cdot,\cdot)\|_{\bH_p^n(t)}^p
            +E \int_t^T\|u(s,\cdot)\|^p_{H_{p}^{n}}ds   \right)  \\
    \leq &C(T,p,d,\lambda,K_0,\Lambda,\varrho(\eps))
        \left(\|F(0,0,\cdot,\cdot)\|_{\bH_p^n(t)}^p
            +\int_t^T\|u\|_{\sH_{p}^{n+2}(s)}^p ds   \right).
  \end{split}
\end{eqnarray*}
Using Gronwall inequality, we get the desired estimation
\eqref{perturbation estimate}.

\emph{Step 2.} We use the standard method of continuity to prove the
existence of the solution $u\in \sH_p^{n+2}$. For $\theta\in [0,1],$ we
consider the BSPDE

\begin{equation}\label{perturbation continuity method 1}
  \left\{\begin{array}{l}
    \begin{split}
  -du =(\cL_{\theta}u+\cM_{\theta}^k v^k+(1-\theta)F(u,v,t,x))dt
                -v^kdW^k_t
    \end{split}\\
    \begin{split}
    u(T,\cdot)=0
    \end{split}
  \end{array}\right.
\end{equation}
where
$$\cL_{\theta}:=\theta \bar{\cL}+(1-\theta)\cL  \textrm{ and }
 \cM_{\theta}^k=\theta \bar{\cM}^k +(1-\theta)\cM^k.  $$
Note that the priori estimate \eqref{perturbation estimate} holds with the
constant $C$ being independent of $\theta.$ Assume that BSPDE
\eqref{perturbation continuity method 1} has a unique solution $ u\in
\sH_{p,0}^{n+2} $ for $\theta=\theta_0.$ Theorem \ref{thm independent
space} shows that this assumption is true for $\theta_0=1.$
For any $u_1\in\sH_{p,0}^{n+2},$  the following BSPDE
\begin{equation}\label{100331}
  \left\{\begin{array}{l}
    \begin{split}
-du =&\{\cL_{\theta_0}u+\cM_{\theta_0}^k v^k+(1-\theta_0)F(u,v,t,x)
    + (\theta-\theta_0) [ (\bar{\cL}-\cL) u_1   \\
        &+(\bar{\cM}^k-\cM^k)(\bD u_1)^k+F(u_1,\bD u_1,t,x) ]   \}dt
                -v^kdW^k_t,
    \end{split}\\
    \begin{split}
    u(T,\cdot)=0,
    \end{split}
  \end{array}\right.
\end{equation}
has a unique solution $u$ in $\sH_{p,0}^{n+2},$ and we can define the
solution map as follows
$$\mathfrak{R}_{\theta_0}:~\sH_{p,0}^{n+2}\rightarrow
\sH_{p,0}^{n+2},~~u_1\mapsto u.$$

Then for any $u_i\in \sH_{p,0}^{n+2},$ $i=1,2,$  we have
\begin{eqnarray*}
  \begin{split}
    \|\mathfrak{R}_{\theta_0}u_2-\mathfrak{R}_{\theta_0}u_1\|_{\sH_p^{n+2}}
    \leq &
        C |\theta-\theta_0|
        \|(\bar{\cL}-\cL)(u_2-u_1)+
            (\bar{\cM}^k-\cM^k)(\bD u_2-\bD u_1)^k        \\
    &+F(u_2,\bD u_2,t,x)-F(u_1,\bD u_1,t,x)\|_{\bH_p^n}    \\
    \leq &
    \bar{C}|\theta-\theta_0|
        \|u_1-u_2\|_{\sH_p^{n+2}},
  \end{split}
\end{eqnarray*}
where $\bar{C}$ does not depend on  $~\theta $ and $\theta_0.$ If
$\bar{C}|\theta-\theta_0|<1/2,$ $\mathfrak{R}_{\theta_0}$ is a
contraction mapping and it
has a unique fixed point $u\in \sH_{p,0}^{n+2}$
which solves BSPDE \eqref{perturbation continuity method 1}.
 In this way if \eqref{perturbation continuity method 1} is
 solvable for $\theta_0,$ then it is solvable for $\theta$ satisfying
 $\bar{C}|\theta-\theta_0|<1/2.$ In finite number of steps starting
  from $\theta=1,$ we get to $\theta=0.$ This completes the proof.
\end{proof}

\begin{lem}\label{lem after perturbation}
    Under the assumptions \ref{ass1}-\ref{ass4}, there exists an
    $\eps=\eps(n,\gamma,d,p,\lambda,\Lambda)>0$ such that
    if $\kappa(\infty -)< \eps,$    then the condition of
    Theorem \ref{thm perturbation} is satisfied .
    Hence by Theorem \ref{thm perturbation}, we conclude that there exists a
    unique solution $u\in \sH_{p,0}^{n+2}$ to BSPDE \eqref{BSPDE} with the zero
    terminal condition satisfying the following inequality
    $$\|u\|_{\sH_p^{n+2}}\leq C(T,\varrho,d,p,
    \lambda,\Lambda)
                \|F(0,0,\cdot,\cdot)\|_{\bH_p^{n}}.$$
\end{lem}
\begin{proof}
Define $\bar{a}(t)=a(t,0)$ and $\bar{\sigma}(t)=\sigma(t,0).$  It
follows from Lemma \ref{lem multiplier} that, for any $(u_{1},v_{1})\in
H_{p}^{n+2}\times H_{p}^{n+1},$ we have
\begin{eqnarray}\label{1003_29}
  \begin{split}
    &\|(a-\bar{a})^{ij}(t,\cdot)(u_1)_{ij}\|_{n,p}+
       \|(\sigma-\bar{\sigma})^{ik}(t,\cdot)(v_1)^k_{i}\|_{n,p} \\
        \leq
      &  C(n,d,p,\gamma)\left(
      \|(a-\bar{a})(t,\cdot)\|_{B^{|n|+\gamma}}\|(u_1)_{xx}\|_{n,p}
       + \|(\sigma-\bar{\sigma})(t,\cdot)\|_{B^{|n|+\gamma}}\|(v_1)_{x}\|_{n,p}
       \right).
  \end{split}
\end{eqnarray}
In view of \eqref{1003_3}, there exists a constant
$\eps_1=\eps_1(n,\gamma,d,p,\lambda,\Lambda)$ such that  if
\begin{eqnarray}\label{100329 1}
  \begin{split}
\|a(t,\cdot)-\bar{a}(t)\|_{B^{|n|+\gamma}}\|(u_1)_{xx}\|_{n,p}
       + \|\sigma(t,\cdot)-\bar{\sigma}(t)\|_{B^{|n|+\gamma}}\|(v_1)_{x}\|_{n,p}
       \leq \eps_1,\forall t\in [0,T],
  \end{split}
\end{eqnarray}
  the condition \eqref{1003_3} in Theorem \ref{thm perturbation} is satisfied.
With a standard method (c.f. \cite[Lemma 6.6, pp. 215-216]{Krylov_99}), we
can check that if if $\eps$ in our lemma is sufficiently small,
\eqref{100329 1} holds true. This complete the proof.
\end{proof}

To prove Theorem \ref{thm general case}, we need a generalization of the
Littlewood-Paley inequality, which is due to Krylov \cite{Krylov_94}.

\begin{lem}\label{Harmonic result 2}
    Let $p\in (1,\infty),~n\in(-\infty,+\infty),$ $\delta>0,$
 and $\zeta_k\in C^{\infty},~k=1,2,3,\dots$ Assume that for any multi-index
 $\alpha$ and $x\in \bR^d$
 $$\sup_{x\in \bR^d}\sum_k|D^{\alpha}\zeta_k(x)|\leq M(\alpha),$$
 where $M(\alpha)$ is constant. Then there exists a constant $C=C(d,n,M)$
 such that, for any $f\in H_p^n,$
 $$\sum_k \|\zeta_k f \|_{n,p}^p\leq C\|f\|_{n,p}^p.$$
 If in addition
    $$\sum_k |\zeta_k(x)|^p\geq \delta,$$
    then for any $f\in H_p^n,$
       $$\|f\|_{n,p}^p \leq C(d,n,M,\delta) \sum_k\|\zeta_k f\|_{n,p}^p.$$
\end{lem}

\begin{proof}[Proof of Theorem \ref{thm general case}]
\emph{Step 1.}
Without loss of generality, assume that $G=0.$

    In fact, by Theorem \ref{thm independent space}, there exists a unique
    solution $\bar{u} \in \sH_p^{n+2} $ for the equation
\begin{equation}
  \left\{\begin{array}{l}
    -du=\Delta udt-v^kdW^k_t,~t\in [0,T];\\
    u(T,x)=G(x)
  \end{array}\right.
\end{equation}
satisfying the estimate
$$\|\bar{u}\|_{\sH_p^{n+2}}\leq C(T,p,d,\lambda,\Lambda)
    \|G\|_{L^{p}(\Omega,\sF_{T},H_{p}^{n+1})}.$$
    Without lose of generality, we consider
$(\bar{u}(t,\cdot),\bD \bar{u}(t,\cdot)) \in H_p^{n+2}\times H_p^{n+1}$ for
any $(t,\omega)\in [0,T]\times \Omega.$ Setting
$(u,v):=(\tilde{u}+\bar{u},\tilde{v}+\bar{v}),$ we need only to consider
the BSPDE
\begin{eqnarray*}
  \begin{split}
    -d\tilde{u}(t,x)=&[a^{ij}(t,x)\tilde{u}_{x^ix^j}(t,x)+\sigma^{ik}(t,x)\tilde{v}_{x^i}^k(t,x)
                +\bar{F}(\tilde{u},\tilde{v},t,x)]dt                        \\
              &~  -\tilde{v}^k(t,x)dW^k_t
  \end{split}
\end{eqnarray*}
where
$$\bar{F}(\tilde{u},\tilde{v},t,x)=F(\tilde{u}+\bar{u},\tilde{v}+\bar{v},t,x)
+a^{ij}(t,x)\bar{u}_{x^i x^j}(t,x)
        +\sigma^{ik}(t,x)\bar{v}^k_{x^i}(t,x)-\Delta \bar{u}(t,x). $$
It can be checked that $\bar{F}$ satisfies the same condition as $F.$

\emph{Step 2.}
We  give a priori estimate for the solution $u\in\sH_{p,0}^{n+2}$ to BSPDE
 \eqref{BSPDE}.

For $\eps>0$ in Lemma \ref{lem after perturbation}, by Assumption
\ref{ass2}, there exists  $\eps_0>0$ such that $\kappa(s)<\eps$ for any
$s\in [0,\eps_0].$ Let $\{ \zeta_l: l=1,2,3,\dots \} $ be a standard
partition of unity in $\bR^d$ such that, for any $l$, the support of
$\zeta_l$ lies in the ball $B(x_l,\eps_0/4).$ For any $l,$ take a function
$\eta_l\in C^{\infty}_{c}$ valued in $[0,1]$ such that the support of
$\eta_l$ lies in  $B_l(x_l,\eps_0/2)$ and $\eta_l=1$ on $B_l.$ Denote
$v:=\bD u.$ Then we get
\begin{equation}
  \left\{\begin{array}{l}
  \begin{split}
    -d(u\zeta_l)(t,x)=&[\tilde{\cL}_l(t,x) (\zeta_l u)(t,x)+\tilde{\cM}^k_l
            (t,x)(\zeta_l v^k)(t,x)+\tilde{F}(t,x)]dt   \\
            &-\zeta_l (x)v^k(t,x)dW^k_t
            \\
    (u\zeta_l)(T,x)=&0
    \end{split}
  \end{array}\right.
\end{equation}
where
\begin{eqnarray*}
  \begin{split}
    \tilde{\cL}_l(t,x) :=&\eta_l(x) \cL (t,x)+(1-\eta_l(x))\cL(t,x_l)       ,\\
    \tilde{\cM}^k_l(t,x):=&\eta_l(x) \cM^k (t,x)+(1-\eta_l(x))\cM^k(t,x_l)  ,\\
    \tilde{F}(t,x):=&-2(\zeta _l)_{x^i}a^{ij}u_{x^j}(t,x)-(\zeta
        _l)_{x^ix^j}a^{ij}u(t,x)-    \\
        &(\zeta_l)_{x^i}\sigma^{ik}v^k(t,x)
            +\zeta_l F(u,v,t,x).
  \end{split}
\end{eqnarray*}
From Theorem \ref{stochastic banach property thm} and Lemma \ref{lem after
perturbation}, we get
$$    \|u\zeta_l\|_{\bH_p^{n+2}}+\|v\zeta_l\|_{\bH_p^{n+1}}\leq
        C(T,\varrho,\lambda,\Lambda,d,p)\|\tilde{F}\|_{\bH_p^n}         \\
    $$
 Applying Lemma
\ref{Harmonic result 2} and \ref{lem multiplier}, we can get such
conclusions as
\begin{eqnarray*}
  \begin{split}
  \sum_l\|\zeta_l F(\omega,t)\|_{n,p}^p\leq
   & C\| F(\omega,t)\|_{n,p}^p,  \\
  \sum_l\|(\zeta_l)_{x^i x^j}a^{ij}u(\omega,t)\|_{n,p}^p
  \leq & C \|a^{ij}u(\omega,t)\|_{n,p}^p \leq C \|u\|_{n,p}^p,   \\
  \|u(\omega,t)\|_{n,p}\leq C \sum_l\|\zeta_lu(\omega,t)\|_{n,p}
  \leq C &\|u(\omega,t)\|_{n,p},~(\omega,t) \in \Omega\times [0,T] a.e..
  \end{split}
\end{eqnarray*}
Integrating each term on $\Omega\times [0,T],$ we have
\begin{eqnarray*}
  \begin{split}
    &\|u\|_{\sH_p^{n+2}}                                   \\
    \leq\,& C(T,n,\kappa,d,p,\lambda,\Lambda)
            \left (\|F(u,v,\cdot,\cdot)\|_{\bH_p^n}+
                \|u\|_{\bH_p^{n+1}}+\|v\|_{\bH_p^n}\right)       \\
        \leq\,& C_1(T,\kappa,n,d,p,\lambda,\Lambda)
                \left(\eps_1\|u\|_{\sH_p^{n+2}}+
                \|F(0,0,\cdot,\cdot)\|_{\bH_p^n}+(1+\varrho(\eps_1))\|u\|_{\bH_p^{n}}+\|v\|_{\bH_p^n}\right)
  \end{split}
\end{eqnarray*}
where  $\eps_1>0$ is arbitrary. Then, noting that $\|v\|_{\bH_p^n} \leq
T^{(2-p)/2}\|v\|_{\bH_{p,2}^n},$ from Lemma \ref{lem BSDE diffusion term
estimate}, we conclude that for any $\eps_2>0,$
 there exists a constant $C=C(T,p,\eps_2,\varrho)$ such that
    $$\|v\|_{\bH_{p}^n}   \leq
        \eps_2
        \left(\|u\|_{\sH_p^{n+2}}+\|F(0,0,\cdot,\cdot)\|_{\bH_p^n}\right)
            +C_2(T,p,\eps_2,\varrho,\Lambda)\|u\|_{\bH_p^n}.
        $$
By choosing $\eps_1+\eps_2$ small enough such that
$C_1(T,\kappa,n,d,p,\lambda,\Lambda)(\eps_1+\eps_2)<1/2, $ we get
\begin{eqnarray}\label{index 1 of prof gneral case thm}
 \|u\|_{\sH_p^{n+2}}\leq
    C(T,\kappa,\varrho,n,d,p,\lambda,\Lambda)
    \left(\|F(0,0,\cdot,\cdot)\|_{\bH_p^n}+\|u\|_{\bH_p^n}\right).
\end{eqnarray}
In view of Theorem
 \ref{stochastic banach property thm} and Remark \ref{rmk sotchastic banach
 thm},
   we can show in a similar way the following inequality
 $$E\|u(t,\cdot)\|^p_{n,p}\leq C\|F(0,0,\cdot,\cdot)\|^p_{\bH_p^n}
    +C\int_t^TE\|u(s,\cdot)\|^p_{n,p}ds
     $$
for all $t\in[0,T].$ Using Gronwall's inequality, we have
$$ \| u\|^p_{\bH_p^n} \leq C \|F(0,0,\cdot,\cdot)\|^p_{\bH_p^n} ,$$
which along with \eqref{index 1 of prof gneral case thm} implies the
following estimate
\begin{eqnarray}\label{priori estimate for general}
  \begin{split}
 \|u\|_{\sH_p^{n+2}}\leq
 C(T,\kappa,\varrho,n,d,p,\lambda,\Lambda)
    \|F(0,0,\cdot,\cdot)\|_{\bH_p^n}.
  \end{split}
\end{eqnarray}

\emph{Step 3.} In the end, proceeding identically as in $Step~ 2$ in the
proof of Theorem \ref{thm perturbation}, we can prove the existence and
uniqueness of the solution. The proof is complete.

\end{proof}

\begin{cor}\label{cor p and q}
  Let the assumptions of Theorem \ref{thm general case} be satisfied. We assume that
  the assumptions are not only satisfied for $p$ but also for $q\in(1,2].$
  Then the solution $u$ in Theorem \ref{thm general case}
  belongs to $\sH_q^{n+2}.$
\end{cor}

\begin{proof}
%
We can
   prove our corollary by completing the $Step~ 3$ of the proof of
   Theorem \ref{thm general case}. The difference from $Step~ 2$ in the proof
of Theorem \ref{thm perturbation} lies that we use the Picard iteration
 this time instead of the contraction mapping principle.
Indeed, consider the equation \eqref{perturbation continuity method 1},
Take a $\theta=\theta_0$ equation \eqref{perturbation continuity method 1}
with zero terminal condition has a unique solution $ u\in
\sH_{p,0}^{n+2}(T)\cap\sH_{q,0}^{n+2}.$ By the way, this assumption is
satisfied for $\theta_0=1$ by Theorem \ref{thm independent space} and
Remark \ref{rmk p q}. Set $u_0=0$ and take iterations
$u_l=\mathfrak{R}_{\theta_0}u_{l-1},$ $l=1,2,3,\dots$. Then there exists a
 constant $\delta>0$ independent of $\theta_0$ such that if $\theta\in
[\theta_0-\delta,\theta_0+\delta]\cap [0,1],$ $u_l$ is a cauchy sequence
both in $\sH_{p,0}^{n+2}$ and $\sH_{p,0}^{n+2}$ and for these $\theta$s
 the solutions in $\sH_{p,0}^{n+2}$ and $\sH_{p,0}^{n+2}$ coincide. In
 finite steps from $\theta=1$, we get to $\theta=0.$ This completes the
 proof.
 \end{proof}

\section{Two related topics}

The proofs of the following results are similar to that of the SPDE
in \cite{Krylov_99}, and will be sketched only.

\subsection{Comparison theorem}
The following theorem shows that the solution to BSPDE \eqref{BSPDE}
is continuous w.r.t. the leading coefficients $a^{ij}$ and
$\sigma^{ik},$ the non-homogeneous drift term $F,$ and the terminal
value $G.$

\begin{thm}\label{thm depending on parameters}
Assume that for $l=1,2,3,\dots,$ we are given $a_l^{ij},\sigma_l^{ik},F_l,$
and $G_l$  verifying the same assumptions as $a^{ij},\sigma^{ik},F$ and $G$
in Theorem \ref{thm general case} with the same constants $\lambda,\Lambda$
and the same functions $\kappa,$ $\varrho.$ Let $\zeta(x)$ be a real
function taking values in $[0,1]$ such that $\zeta(x)=1$ if $|x|\leq 1$ and
$\zeta(x)=0$ if $|x|\geq 2.$ Define $\zeta_r(x)=\zeta(x/r)$ for
$r=1,2,3,\dots.$ And we also assume that, for
$r=1,2,3,\dots,i,j=1,\dots,d,k=1,\dots,m,t\in [0,T] ,$ and
$\omega\in\Omega,$
\begin{eqnarray}\label{thm depending on parameters 1}
  \begin{split}
    \|\zeta_r\{ a^{ij}(t,\cdot)-a_l^{ij}(t,\cdot)\} \|_{n,p}
        +\|\zeta_r\{ \sigma^{ik}(t,\cdot)-\sigma^{ik}_l(t,\cdot)\} \|_{n,p}
            \rightarrow 0
  \end{split}
\end{eqnarray}
as $l\rightarrow \infty.$ Furthermore,
$E\|G_l-G\|^p_{n+1,p}\rightarrow 0$ and
\begin{eqnarray}\label{thm depending on parameters 2}
    \|F(u,v,\cdot,\cdot)-F_l(u,v,\cdot,\cdot)\|_{\bH_p^{n}}
        \rightarrow 0,
\end{eqnarray}
whenever $u\in \sH_p^{n+2}$ and $v:=\bD u$. If we take the function $u$
from Theorem \ref{thm general case} and for any $l$ define
$u_l\in\sH_p^{n+2}$ as the unique solution of the following BSPDE
\begin{equation}\label{thm depending on parameters 3}
  \left\{\begin{array}{l}
    \begin{split}
      -d u_l (t,x)=&\big[
      a_l^{ij}(t,x)u_{lx^i x^j}(t,x) +\sigma_l^{ik}(t,x)v^{k}_{lx^i}(t,x)
      +F_l(u_l,v_l,t,x)\big]dt\\
      &-v_l^k(t,x)dW^k_t,\quad~~~ (t,x)\in [0,T]\times \bR^{d},
    \end{split}\\
    \begin{split}
    u_l(T,x)=G_l(x),~~~~~~x\in \bR^{d},
    \end{split}
  \end{array}\right.
\end{equation}
where $v_l:=\bD u_l,$ then we have $\|u-u_l\|_{\sH_p^{n+2}}\rightarrow 0$
as $l\rightarrow \infty.$
\end{thm}

\begin{proof}
  Let $\bar{u}_l=u-u_l$ and $\bar{v}_l=v-v_l.$ Then we have
    \begin{equation}\label{thm depending on parameters 3}
  \left\{\begin{array}{l}
    \begin{split}
      -d \bar{u}_l (t,x)=&\big[
      a_l^{ij}(t,x)\bar{u}_{lx^i x^j}(t,x) +\sigma_l^{ik}(t,x)\bar{v}^{k}_{lx^i}(t,x)
      +f_l(\bar{u}_l,\bar{v}_l)\big]dt\\
      &-\bar{v}_l^k(t,x)dW^k_t,\quad~~~ (t,x)\in [0,T]\times \bR^{d},
    \end{split}\\
    \begin{split}
    \bar{u}_l(T,x)=\bar{G}_l(x),~~~~~~x\in \bR^{d},
    \end{split}
  \end{array}\right.
\end{equation}
where
$$ f_l(\bar{u}_l,\bar{v}_l)=(a^{ij}-a_l^{ij})u_{x^i x^j}
    +(\sigma^{ik}-\sigma^{ik}_l)v_{x^i}+F(u,v)-F_l(u-\bar{u},v-\bar{v}).
$$
Then by Theorem \ref{thm general case}, we obtain
$$\|u-u_l\|_{\sH_p^{n+2}}\leq C J_l,$$
where $C$ is independent of $l$ and
\begin{eqnarray}
  \begin{split}
J_l=&\|(a^{ij}-a_l^{ij})u_{x^i x^j}\|_{\bH_p^n}
    +\|(\sigma^{ik}-\sigma^{ik}_l)v_{x^i}\|_{\bH_p^n}+\\
    &\|F(u,v)-F_l(u,v)\|_{\bH_p^n}
    +\big(E\|G_l-G\|^p_{n+1,p}\big)^{1/p}.
  \end{split}
\end{eqnarray}
By our assumptions, we have
\begin{eqnarray}
  \begin{split}
    \limsup_{l\rightarrow \infty}J_l\leq \limsup_{l\rightarrow
    \infty}\{ \|(a^{ij}-a_l^{ij})u_{x^i x^j}\|_{\bH_p^n}
    +\|(\sigma^{ik}-\sigma^{ik}_l)v_{x^i}\|_{\bH_p^n}   \}.
  \end{split}
\end{eqnarray}

Then the following is standard (for conference, see the proof of Theorem
5.7 of \cite{Krylov_99} pp.209-210).

 For any $\phi\in C^{\infty}_{c},$ let $r$ be so large that $\phi
\zeta_r=\phi.$ Then, by Lemma \ref{lem multiplier}, we get
\begin{eqnarray}
  \begin{split}
    \|(a^{ij}-a_l^{ij})u_{x^i x^j}\|_{n,p}\leq C
    \|(u-\phi)_{x^ix^j}\|_{n,p}
    +\|(a^{ij}-a_l^{ij})\phi_{x^ix^j}\|_{n,p},
  \end{split}
\end{eqnarray}
$$\|(a^{ij}-a_l^{ij})\phi_{x^ix^j}\|_{n,p}
    =\|(a^{ij}-a_l^{ij})\zeta_r\phi_{x^ix^j}\|_{n,p}
    \leq
    C\|(a^{ij}-a_l^{ij})\zeta_r\|_{n,p}\|\phi\|_{B^{|n|+2+\gamma}},
    $$
where the constants $C$'s are independent of $r$ and $l$. Thus,
$$\limsup_{l\rightarrow
    \infty} \|(a^{ij}-a_l^{ij})u_{x^i x^j}\|_{n,p}\leq C
    \|(u-\phi)_{x^ix^j}\|_{n,p}
   \textrm{ for }(t,\omega)\in [0.T]\times\Omega,a.e., $$
and the arbitrariness of $\phi$ implies the left-hand side above is zero.
Then by Lemma \ref{lem multiplier} and the dominated convergence theorem,
we conclude that
$$\lim_{l\rightarrow \infty} \|(a^{ij}-a_l^{ij})u_{x^i x^j}\|_{\bH_p^n}
        =0.$$
Similarly, we can get $\lim_{l\rightarrow \infty}
\|(\sigma^{ik}-\sigma^{ik}_l)v_{x^i}\|_{\bH_p^n}
        =0.$
\end{proof}
\begin{rmk}
  From Lemma \ref{lem finite approximation H_pn}, it follows that
  the condition \eqref{thm depending on parameters 2} holds for any $u\in \sH_p^{n+2}$
   if and only if it is satisfied for $u(t,x)\equiv\phi,v^k(t,x)\equiv
   \phi^k$ with any $\phi,\phi^k\in C^{\infty}_{c},k=1,\dots,m.$
\end{rmk}

\begin{cor}\label{lem smooth coefficients}
Take $\zeta_l$ from Lemma \ref{lem multiplier}. Under the assumptions of
Theorem \ref{thm general case}, for $l=1,2,3,\dots,$ we define
$$(a_l,\sigma_l)=(a,\sigma)(t,\cdot)\ast \zeta_l(x),
        G_l=G\ast \zeta_l(x),
   $$
and also
$$F_l(u,v,t,x)=F(u,v,t,\cdot)\ast \zeta_l(x)
   =\int_{\bR^d}F(u(x),v(x),t,x-y)\zeta_l(y)dy .$$ Then the
assumptions of Theorem \ref{thm depending on parameters} are satisfied, and
if we take $u_l\in\sH_p^{n+2} $ as the unique solution of BSPDE \eqref{thm
depending on parameters 3}, we have $\|u-u_l\|_{\sH_p^{n+2}}\rightarrow 0$
as $l\rightarrow \infty.$
\end{cor}

 As the proof of the corollary is just a verification, which is very similar to
  \cite[Corollary 5.10]{Krylov_99}, it is omitted here.

\begin{thm}
  Under the assumptions of Theorem \ref{thm general case}, let $u$ be
  the solution of BSPDE \eqref{BSPDE} for $n=0.$ And further, assume that
  $$ F(u,v,t,x)=b^i(t,x)u_{x^i}+c_0(t,x)u(t,x)+c_k(t,x)v^k(t,x)+f(t,x),
  $$
  where $b^i(t,x),c_0(t,x),c_k(t,x),k=1,\dots,m$ are bounded $\sP\times\cB(\bR^d)$
  -measurable functions on $[0,T]\times \Omega \times \bR^d$ and $f(t,x)\geq 0.$ Also
  assume that $G\geq 0$ almost surely. Then $u(t,\cdot)\geq 0$ for
  all $t \in [0,T]$ almost surely.
\end{thm}
\begin{proof}
    First, we take two nonnegative sequences
    $(f^l)_{l\geq 1}$ in $L^{\infty}(\Omega\times[0,T],\sP,H^0_2)
    \cap \bH^0_p$  and $(G^l)_{l\geq 1} $ in $L^2(\Omega,\sF_T,H^1_2)\cap
    L^p(\Omega,\sF_T,H^1_p)$
    such that $\|f^l-f\|_{\bH^0_p}\rightarrow 0 $ and
    $\|G^l-G\|_{L^p(\Omega,\sF_T,H^1_p)}\rightarrow 0 $ as
    $l\rightarrow \infty.$ Next, Corollary \ref{lem smooth coefficients} allows
     us to assume  that $G^l,$ $f^l$ and all the other coefficients are infinitely
   differentiable in $x.$

    After those above, by Theorem \ref{thm depending on parameters}
     we get an approximating solutions $u^l$ of $u.$ In this case
     the  assumptions of Theorem \ref{thm general case} are satisfied for $p=2,$
    and any $n\geq 0.$ Then, Corollary \ref{cor p and q}, yields $u^l\in
     \sH^r_2$ for any $r\geq 0.$ Furthermore, in this case the assumptions
      of \cite[Theorem 5.1]{DuMeng09}, \cite[Theorem 6.1]{Hu_Ma_Yong02} and
      \cite[Theorem 6.1]{Tang_05} are all satisfied, and the comparison
      theorems
     there all imply $u^l\geq 0$ $(a.e. ~(t,x,\omega)).$
    By taking limits, we get $u\geq 0$ $(a.e. ~(t,x,\omega)).$
    On the other hand, in light of Lemma \ref{lem BSDE}, it
    follows that $u\in C([0,T],H_p^0)$ $a.s.$, which implies
    $u\geq 0$ (at least for a modification of $u$) for all $t\in [0,T]$ almost surely.
\end{proof}

\subsection{$L^p$ theory for $p>2$}
When $p<2$, the assertion of Lemma \ref{Harmonic result 1} is not
true in general. This fact makes the $L^p$-theory we have
established in Section 5 require the assumption $p\in (1,2]$ and
Krylov's seminal work (\cite{Kryl96,Krylov_99}) require $p\in
[2,\infty).$ However, if we consider SPDEs~\eqref{p>2 simple spde}
where the diffusion is homogeneous in the unknown variable, the
harmonic result (Lemma \ref{Harmonic result 1}) can be avoided,
which could allow us to get further results.

Consider the following BSPDE
\begin{equation}\label{p>2 BSPDE}
  \left\{\begin{array}{l}
    \begin{split}
      -d u (t,x)=&\big[
      a^{ij}(t,x)u_{x^i x^j}(t,x) +\sigma^{ik}(t)v^{k}_{x^i}(t,x)
      +F(u,\sigma^i u_{x^i}+v,t,x)\big]dt\\
      &-v^k(t,x)dW^k_t,\quad~~~ (t,x)\in [0,T]\times \bR^{d};
    \end{split}\\
    \begin{split}
    u(T,x)=G(x),~~~~~~x\in \bR^{d}.
    \end{split}
  \end{array}\right.
\end{equation}
\begin{defn}\label{p>2 defn}
  We call $(u,v)$ a solution pair of BSPDE \eqref{p>2 BSPDE} in
  $\bH_p^{n+2}\times\bH_{p,2}^{n}$ if $u\in\bH_p^n,$
  $v(\cdot,\cdot+\int_0^{\cdot}\sigma^k(s)dW_s^k)\in \bH_{p,2}^n$
  and for any $\phi\in C^{\infty}_{c},$ the
  equality
\begin{equation}\label{p>2 defn equation}
  \begin{split}
  (u(\tau,\cdot),\phi)=&(G,\phi)+\int_{\tau}^{T}( a^{ij}(t,\cdot)u_{x^i x^j}(t,\cdot) +
       \sigma^{ik}(t)v^{k}_{x^i}(t,\cdot)
          +F(u,\sigma^i u_{x^i}+v,t,\cdot),\phi)dt\\
  &-\int_{\tau}^{T}(v^k(t,\cdot),\phi)dW^k_t,
  \quad ~\forall (t, \phi) \in [0,T)\times C^{\infty}_{c}
  \end{split}
\end{equation}
holds for all $\tau\in [0,T]$ with probability 1.
\end{defn}
For the case $p>2,$ we have presented some results in Remark
\ref{p>2 simple case} on BSPDEs with constant-field-valued
coefficients. Through a procedure similar to the case $p\in(1,2]$ we
get the following result.

\begin{prop}
  For $p>2$ and $n\in\bR,$
  suppose that $a$ and $\sigma$ satisfy  Assumption \ref{ass1}-\ref{ass3} with
  $\sigma$ being invariant in the space variable. Let
  $F(0,0,\cdot,\cdot)\in\bH_p^n.$ For any $(h,g)\in\bH_p^{n+2}\times\bH_{p,2}^n,$
  $F(h,g,t,\cdot)$ is an $H_p^n$-valued $\sP$-measurable process such that
  there is a continuous and decreasing function
 $\varrho:(0,\infty)\rrow[0,\infty)$ and a constant $L>0$ such that
 for any $\eps>0,$   we have
\begin{equation}
  \begin{split}
  &\|F(\bar{h},\bar{g},\cdot,\cdot)-F(h',g',\cdot,\cdot)\|_{\bH_p^n(t)}  \\
  \leq &\eps \|\bar{h}-h'\|_{\bH_p^{n+2}(t)}+
        \varrho (\eps)\|\bar{h}-h'\|_{\bH_p^{n}(t)}+L\|\bar{g}-g'\|_{\bH_{p,2}^n(t)}
        ,    \\
         &\quad ~\bar{h},h'\in \bH_{p}^{n+2} \textrm{ and } \bar{g},g'\in \bH_{p,2}^{n},
  \end{split}
\end{equation}
 holds for any $t\in [0,T).$
 Consider $G\in L^p(\Omega,\sF_T,H_p^{n+1}).$ Then BSPDE \eqref{p>2 BSPDE}
 has a unique solution pair $(u,v)$ in $\bH_p^{n+2}\times\bH_{p,2}^n.$ For
 this solution pair, we have
 $$\|u\|_{\bH_p^{n+2}}+\|v'\|_{\bH_{p,2}^n}\leq
    C(T,n,\kappa,\varrho,d,p,\lambda,\Lambda)
        \left(\|F(0,0,\cdot,\cdot)\|_{\bH_p^n}+\|G\|_{L^p(\Omega,\sF_T,H_p^{n+1})}
          \right)$$
 where $v'(t,x)=v(t,x+\int_0^t\sigma^k(s)dW_s^k),$ $(t,x)\in [0,T]\times\bR^d. $
\end{prop}
Here, we only give a sketch of the proof.
 First, take $\zeta(t,x)=u(t,x+\int_0^t\sigma^k(s)dW_s^k).$
 Applying formally the It\^o-Wentzell formula (c.f \cite{Krylov_09}),
   we can rewrite
 the BSPDE \eqref{p>2 BSPDE}
\begin{equation}
  \left\{\begin{array}{l}
    \begin{split}
      -d \zeta (t,x)=&\big[
      \bar{a}^{ij}(t,x)\zeta_{x^i x^j}(t,x)
      +F(\zeta,\sigma^i\zeta_{x^i}+ v',t,x+\int_0^t\sigma^k(s)dW_s^k)\big]dt\\
      &-(\sigma^{ki}\zeta_{x^i}(t,x)+ v'^k(t,x))dW^k_t,\quad~~~ (t,x)\in [0,T]\times \bR^{d};
    \end{split}\\
    \begin{split}
    \zeta(T,x)=\bar{G}(x),~~~~~~x\in \bR^{d},
    \end{split}
  \end{array}\right.
\end{equation}
where
$\bar{a}(t,x):=a(t,x+\int_0^t\sigma^k(s)dW^k_s)-\frac{1}{2}\sigma\sigma^{\mathcal
{T}}$ and $\bar{G}=G(x+\int_0^T\sigma^k(s)dW^k_s).$ Actually the estimate
about $v$ are deduced from Lemma \ref{lem BSDE}. The proof of the other
assertions are very similar to those seen in Section 5.3.
\bibliographystyle{siam}

\end{document}